\documentclass[11pt]{amsproc}
\pdfoutput=1
\usepackage{amsmath,amssymb,amsthm,eucal, eufrak,amscd,tikz,mathtools}
\usepackage{xcolor}
\usepackage{enumitem}
\usepackage[shortalphabetic]{amsrefs}
\usepackage{tikz-cd}
\usepackage{graphicx}
\definecolor{allrefcolors}{rgb}{0,0.2,0.5}
\usepackage[linktocpage=true,colorlinks=true,allcolors=allrefcolors,bookmarksopen,bookmarksdepth=3]{hyperref}
\usepackage[margin=1.25in]{geometry}
\usepackage[all]{xy}
\usepackage{ stmaryrd }
\usepackage[draft]{say}
\usepackage{comment}

\title{Affine nil-Hecke algebras and Quantum cohomology}
\author{Eduardo Gonz\'alez}
\author{Cheuk Yu Mak}
\author{Dan Pomerleano}

\newtheorem{lem}{Lemma}[section]
\newtheorem{prop}[lem]{Proposition}
\newtheorem{thm}[lem]{Theorem}
\newtheorem{cor}[lem]{Corollary}

\newtheorem{defn}[lem]{Definition}

\newtheorem{rem}[lem]{Remark}
\newtheorem{example}[lem]{Example}

\def\ul{\underline}
\def\End{\operatorname{End}}
\def\vdim{\operatorname{vdim}}
\def\Ham{\operatorname{Ham}}
\def\Hom{\operatorname{Hom}}
\def\im{\operatorname{Im}}

\def\ev{\operatorname{ev}}
\def\ol{\overline}

\def\ol{\overline}
\def\cF{\mathcal{F}}

\def\cS{\mathcal{S}}
\def\cM{\mathcal{M}}
\def\cC{\mathcal{C}}

\def\cQ{\mathcal{Q}}
\def\bk{\mathbf{k}}
\definecolor{cym}{rgb}{.7,0,.4}\newcommand{\cym}{\color{cym}}

\begin{document}
\maketitle

\begin{abstract}
Let $G$ be a compact,  connected Lie group and $T \subset G$ a maximal torus.  Let $(M,\omega)$ be a monotone closed symplectic manifold equipped with a Hamiltonian action of $G$.  We construct a module action of the affine nil-Hecke algebra $\hat{H}_*^{S^1 \times T}(LG/T)$ on the $S^1 \times T$-equivariant quantum cohomology of $M$,  $QH^*_{S^1 \times T}(M).$ Our construction generalizes the theory of shift operators for Hamiltonian torus actions \cite{MR2753265, LJ21}.  We show that,  as in the abelian case,  this action behaves well with respect to the quantum connection.  As an application of our construction,  we show that the $G$-equivariant quantum cohomology $QH_G^*(M)$ defines a canonical holomorphic Lagrangian subvariety $\mathbb{L}_G(M) \hookrightarrow BFM({G^{\vee}_\mathbb{C}})$ in the BFM-space of the Langlands dual group,  confirming an expectation of Teleman from \cite{Teleman2}.   

\end{abstract}


\section{Introduction} 

Let $(M^{2n},\omega)$ be a monotone closed symplectic manifold,  that is $[\omega]=\lambda [c_1(M)] \in H^2(M)$ with $\lambda>0$,  equipped with a Hamiltonian action of a torus $T$.  Let $QH^*_{S^1 \times T}(M)$ denote the quantum cohomology of $M$ which is equivariant with respect to the $T$-action and loop rotation. As a vector space this is given by \begin{align*} QH^*_{S^1 \times T}(M):= H^*_{T}(M)[q^{\pm 1},u] \end{align*}
where $q$ is the Novikov variable, and $u$ is the positive
generator of $H^*(BS^1)$.  This vector space carries much
structure, the most elementary pieces of which are as
follows: \begin{itemize} \item The reduction modulo $u$ is
	the ordinary \(T\)-equivariant quantum cohomology,  $QH^*_ {T}(M)$,  which carries an equivariant quantum product.  \item The full equivariant quantum cohomology $QH^*_{S^1 \times T}(M)$ carries a quantum connection $\nabla_{q\partial_q}$, which differentiates in the direction of the Novikov variable. \end{itemize} 
 For every co-character $\sigma: S^1 \to T$,  \cite{MR2753265, LJ21} define a $\mathbf{k}[q^{\pm}]$-linear operator \begin{align*} S_\sigma: QH^*_{S^1 \times T}(M) \to QH^*_{S^1 \times T}(M) \end{align*} 
known as a shift operator.  These operators are $S^1$-equivariant lifts of the invertible operators of \cite{MR1487754} and have proven to be of central importance in the study of quantum cohomology of symplectic resolutions \cite{BMO, MO} and toric mirror symmetry \cite{Iritani}.  The three main properties of these operators are: \begin{enumerate}[label=(\Alph*)] \item \label{item:A}
$S_{\sigma_{1}} \circ S_{\sigma_{2}} = S_{\sigma_1+\sigma_2}$ \item \label{item:B} $S_{\sigma}$ is a $\sigma$-twisted homomorphism in the sense of \cite[Section 3.1]{Iritani}.  \item \label{item:C}  $S_{\sigma}$ commutes with the  quantum connection.  \end{enumerate}

The purpose of these notes is to: \begin{itemize} \item extend the theory of
shift-operators to the case where a general compact,
connected Lie group $G$ acts on $M$ in a Hamiltonian
fashion.  \item develop a connection between these non-abelian shift operators and well-known structures in geometric representation theory and gauge theory.   \end{itemize} The theory we construct is based on the equivariant topology of the smooth affine flag variety $LG/T$ equipped with its
natural $\hat{T}:=S^1 \times T$ action.  Intuitively,  our construction should be understood as a ``parameterized version" of a shift operator,  where the moduli spaces involved are parameterized  by ($\hat{T}$-equivariant) cycles in $LG/T$.  The relevant cycles in $LG/T$ define homology classes in a variant (reviewed in  \S \ref{section: BMH}) of the usual Borel equivariant homology which we call the ``semi-infinite equivariant" homology.\footnote{In the literature (e.g. \cite{BFM}),  this is referred to as 
equivariant Borel-Moore homology,  however we find this terminology to be potentially ambiguous. } The semi-infinite homology group $\hat{H}_*^{\hat{T}}(LG/T)$ is naturally a module over $H^{*}_{\hat{T}}(pt)$ and carries a convolution algebra structure known as the affine nil-Hecke algebra.  The structure of this ring has been studied extensively by Kostant and Kumar \cite{KK2, KK1,Ku} and has seen important applications to geometric representation theory and combinatorics (for example \cite{Ku2, Dye,Lam, OblomkovCarlsson}).  Our first main result is the following: 

\begin{thm}[see Theorem \ref{t:proof1}]\label{t:main} There is a module action \begin{align} \mathcal{S}: \hat{H}_*^{\hat{T}}(LG/T)\otimes QH^*_{\hat{T}}(M) \to QH^*_{\hat{T}}(M) \label{eq:SeidelThm} \end{align} For any $\beta \in H_{\hat{T}}^*(pt)$ and $\alpha \in \hat{H}^{\hat{T}}_*(L_{poly}G/T)$, we have $\cS(\beta \cdot \alpha,-)=\beta \cdot \cS(\alpha,-)$.  
\end{thm} 

 We give a detailed overview of our construction in \S \ref{ss:overviewSeidelMap} and discuss its relation to other constructions which appear in the literature in \S \ref{section: relatedworks}.  In the case where $G$ is abelian, it is not difficult to see that this module structure is a re-packaging of properties \ref{item:A} and \ref{item:B} of shift-operators.  To keep the notation simple,  let us spell this out when $G=S^1.$ In this case $LS^1/S^1 \cong \Omega S^1$ and the $S^1 \times S^1$ action on $\Omega S^1$ is trivial.  So $\hat{H}_*^{S^1 \times S^1}(\Omega S^1)$ naturally splits and it follows that there is a ring isomorphism $$ \hat{H}_*^{G}(\Omega S^1) \cong H^*(BS^1)\otimes \mathbb{Z}[\pi_1(S^1)] \cong \mathbb{Z}[h,z^{\pm 1}].  $$ The convolution product is given by the deformation (see \cite[Example 7.2]{Teleman1} or \eqref{eq:multipbasis} below) 
$$  \hat{H}_*^{S^1 \times S^1}(LS^1/S^1) \cong  \frac{\mathbb{Z}[u] \langle h,z^{\pm 1} \rangle}{(zh=(h+u)z)} $$ 

Our second main result is the generalization of property \ref{item:C} to the non-abelian context. 

\begin{thm}[see Theorem \ref{t:commute}]\label{t:connection}
 For any $\alpha \in \hat{H}_*^{\hat{T}}(LG/T)$,  \begin{align*} \mathcal{S}_\alpha \circ \nabla_{q\partial_q}= \nabla_{q\partial_q} \circ \mathcal{S}_\alpha \end{align*} \end{thm} 


The first person to consider relations between ($u=0$ limits of) affine nil-Hecke algebras and quantum cohomology was Peterson \cite{Peterson}.  In the special case where $M:=G/P$ is a generalized flag variety,  he conjectured a combinatorially defined formula for a ring homomorphism $\mathcal{P}: \hat{H}_*^{T}(\Omega G) \to QH_T^*(G/P)$,  where $\hat{H}_*^{T}(\Omega G)$ is equipped with the Pontryagin product.  Peterson's formula was later justified by Lam-Shimozono \cite{Lam-Shimozono} and has recently been given a geometric interpretation in the series of works \cite{ChiHong1, ChiHong3,ChiHong2}.   It turns out that a similar homomorphism can be defined for any $M$ using $u=0$ limits of shift-operators.  Namely,  there is a natural $T$-equivariant map $j: \Omega G \to LG/T$,  which induces a map: \begin{align*} j_*: \hat{H}_*^{T}(\Omega G) \to \hat{H}_*^{T}(LG/T).  \end{align*}

Let \begin{equation} \begin{aligned} \label{eq:peterson2} \mathcal{P}: \hat{H}_*^{T}(\Omega G) \to QH_T^*(M),  \\ \alpha_0 \mapsto \mathcal{S}_{j_{*}(\alpha_0)}^{u=0}([M]_T) \end{aligned} \end{equation}

 \begin{lem}[see Lemma \ref{l:Peterson}] Equip $\hat{H}_*^{T}(\Omega G)$ with the $T$-equivariant Pontryagin product and $QH_T^*(M)$ with the equivariant quantum product.  The map \eqref{eq:peterson2} becomes a ring homomorphism.   \end{lem} 

We caution the reader that we do not actually prove that when $M=G/P$,  the map \eqref{eq:peterson2} agrees with Peterson's formula.  However,  we expect that the methods from \cite{ChiHong3} could be used to prove this.  It seems to be an interesting open problem to give an analogous combinatorial formula for the shift-operators on $QH^*_{\hat{T}}(G/P)$.  

A second source of motivation for the present paper comes from work of Teleman \cite{Teleman2, Teleman1} (see also \cite{BDGH} for a physical derivation),  who has proposed a framework connecting $G$-equivariant symplectic topology and Rozansky-Witten theory (the 3D B-model).  Let $G^{\vee}$ be the Langlands dual group of $G$,  and $G_\mathbb{C}^{\vee}$ denote the complexification of $G^{\vee}$.  The BFM-space $BFM(G^{\vee}_\mathbb{C})$ is defined to be the variety formed by pairs $(g,x)$ where $x$ lies in a fixed Kostant slice of the Lie algebra $\mathfrak{g}_\mathbb{C}^{\vee}$,  $g \in G_{\mathbb{C}}^{\vee},$ and $\operatorname{Ad}_g(x)=x$. This is a smooth,  affine holomorphic symplectic manifold.  In fact,  it is hyperk{\"a}hler as can be seen from its alternative description as a moduli space of solutions to Nahm's equation for $G^{\vee}$ \cite[Appendix A]{BFN}.



 Teleman conjectures that a compact symplectic manifold $M$ with Hamiltonian $G$-action should define an object in the Rozansky-Witten 2-category of $BFM(G^{\vee}_\mathbb{C}).$ While this 2-category has yet to be rigorously defined,  part of the data defining such an object is expected to be a holomorphic Lagrangian subvariety in $BFM(G^{\vee}_\mathbb{C})$.   As we now explain,  the theory of shift operators allows for a direct construction of this Lagrangian.  The link to shift operators is provided by a result of Bezrukavnikov-Finkelberg-Mirkovic \cite[Theorem 2.12]{BFM},  who prove that there is an isomorphism of algebraic varieties:  \begin{align*} BFM(G^{\vee}_\mathbb{C}) \cong \operatorname{Spec}(\hat{H}_*^G(\Omega G,\mathbb{C})),  \end{align*} where $\Omega G= LG/G$ is the based loop space equipped with its Pontryagin product.  The semi-infinite equivariant homology $\hat{H}_*^{S^1 \times G}(\Omega G)$ also carries a convolution product,  which gives a deformation quantization of the Pontryagin ring (in the direction of the symplectic structure).  When the ground field $\mathbf{k}$ has characteristic zero,  there is an embedding 
\begin{align*} \hat{H}_*^{S^1 \times G}(\Omega G)
  \hookrightarrow  \hat{H}_*^{{S^1\times T}}(LG/T) \end{align*}
which after identifying $QH^*_{S^1 \times G}(M)= QH^*_{{S^1\times T}}(M)^{W}$,  induces a module structure \begin{align} \label{eq:SGintro} \mathcal{S}_G: \hat{H}_*^{S^1 \times G}(\Omega G)\otimes QH^*_{S^1 \times G}(M) \to QH^*_{S^1 \times G}(M) \end{align} and hence a module structure $\hat{H}_*^{G}(\Omega G)\otimes QH^*_{G}(M)_{|q=1} \to QH^*_{G}(M)_{|q=1}$ by reduction.   

Thus,  we can view $QH^*_{G}(M,\mathbb{C})_{|q=1}$ as defining a coherent sheaf over $BFM(G^{\vee}_\mathbb{C}).$  As noted by Teleman \cite[Remark 2.3]{Teleman1},  the classical theory of modules over a deformation quantization \cite{Gabber} easily implies: 

\begin{cor}\label{c:LagrangianSupport}
Suppose $G$ is a compact,  connected Lie group and let $M$ be as above.  The support of $QH^*_G(M,\mathbb{C})_{|q=1}$ as a coherent sheaf over $BFM(G^{\vee}_\mathbb{C})$ is a (possibly singular) holomorphic Lagrangian subvariety $\mathbb{L}_G(M) \hookrightarrow BFM(G^{\vee}_\mathbb{C}).$   \end{cor}

To illustrate this corollary,  we show (see the end of \S \ref{section: other}) how this Lagrangian subvariety appears in various calculations of quantum cohomology which appear in the literature.  This result provides a conceptual underpinning for the appearance of Lagrangian subvarieties in calculations of quantum cohomology (c.f.  the philosophical question posed on page 613 of \cite{GiventalKim93}).  It is worth noting that Teleman \cite[Section 2]{Teleman2}  envisioned a somewhat different construction of the $\hat{H}_*^{G}(\Omega G)$ module structure on $QH^*_{G}(M)$ based on the ``Fukaya 2-category" of $T^*G$.  Versions of his construction have since been carried out in non-equivariant Floer theory \cite{MR3868001,Oh-Tanaka}.  Generalizing these ideas to $G$-equivariant Floer theory would require developing the foundations of $G$-equivariant (quilted-) Lagrangian Floer theory but could be natural for incorporating chain-level algebraic structures.  It would also be interesting to understand how the quantized module  \eqref{eq:SGintro} could be constructed from this story.

\subsection{Related Constructions} \label{section: relatedworks}

As mentioned above,  our construction involves counts of pseudo-holomorphic sections parameterized by equivariant cycles in $LG/T$.  A closely related idea was first introduced by Savelyev \cite{Sav} to define a version of the Seidel homomorphism which incorporates higher-dimensional cycles of Hamiltonian loops.  Moreover,   Savelyev's applications to Hofer geometry make use of $S^1$-equivariant versions of these moduli spaces (with respect to loop-rotation action on $LHam(M,\omega)$).  The moduli spaces that we consider in this paper are appropriate $S^1 \times T$-equivariant modifications of his constructions.  This additional equivariance leads to new features --- the algebraic structures that we consider (affine nil-Hecke action compatible with connections) together with the resulting links to geometric representation theory are only present when $M$ admits a Hamiltonian $G$-action and one works $S^1 \times T$ (or $S^1 \times G$)-equivariantly.  We also note that the construction of \cite{ChiHong3} also extends Savelyev's work (in the special case where $M=G/P$ and without loop-equivariance) and hence is also closely related to our work.

\subsection{Organization of the paper}

In Section \ref{section: BMH}, we recall the definition of $\hat{H}_*^{\hat{T}}(LG/T)$ and its convolution product structure.  In Section \ref{s:Seidelmor}, we construct the equivariant Seidel morphism from \eqref{eq:SeidelThm}.  Section \ref{s:module} and Section \ref{s:connection} are devoted to the proof of Theorem \ref{t:main} and \ref{t:connection}, respectively.  In Section \ref{section: other}, we recall 
Gabber's theorem on deformation quantization, and then give the proof of Corollary \ref{c:LagrangianSupport}.

\subsection*{Acknowledgements}
  C.M. would like to thank Pavel Safronov for helpful communications.  D.P.  would like to thank Constantin Teleman for his generous and patient explanations of \cite{Teleman2,  Teleman1}.  He would also like to thank Victor Ginzburg for a helpful email exchange.  C.M. is supported by the Simons Collaboration on Homological Mirror Symmetry.   D.P.  was partly supported by the Simons Collaboration in Homological Mirror Symmetry,  Award \# 652299 while working on this project.

\section{Topology background} \label{section: BMH}
\subsection{Semi-infinite homology} 

Let $K$ be a compact, connected Lie group.  We will need to discuss a variant of equivariant homology of $K$-spaces $N$,  which we call semi-infinite homology.  Roughly speaking the theory behaves like
equivariant cohomology in the ``$BK$-directions" and
homology in the ``$N$-directions" and arises naturally when studying Poincar\'e-Lefschetz
 duality in the equivariant context.    This theory was introduced in the algebro-geometric literature \cite{MR1614555,  MR1786481,Graham} and a nice topological reference is \cite{MR3267019,  MR3190596}.  All homology groups below will be taken with respect to some ground field $\bk$.

Let us recall the definitions from \cite{MR1614555, MR1786481,Graham},  which are given in terms of finite dimensional approximations to classifying spaces of compact Lie groups.  Recall that there is a convenient model for the total space of the universal $U(k)$-bundle,  $EU(k)$, given by taking the set of orthormal $k$-frames in a complex Hilbert space $\mathcal{H}.$ Topologically,  this is the limit of spaces of $k$-frames in $\mathbb{C}^n$  as $n \to \infty$:  $$ EU(k)_1 \subset \cdots \subset EU(k)_n \subset \cdots .$$ The group $U(k)$ acts freely on this space and the quotient is the infinite Grassmannian of \(k\)-planes $Gr_k(\mathcal{H}).$ For any vector $\vec{k}=(k_1,\cdots,k_\ell)$, let $$U(\vec{k}):=\prod U(k_i),  EU(\vec{k}):=\prod EU(k_i)$$ For a general compact Lie group $K$, choose a faithful embedding of $K$ into a product of unitary groups (such embeddings exist by the Peter-Weyl theorem and Weyl's unitary trick) $$\rho: K \hookrightarrow U(\vec{k}).$$ Then $K$ acts on $EU(\vec{k})$, we let $BK_{\rho}$ denote the quotient \begin{align} \label{eq:classifying} BK_{\rho}:=EU(\vec{k})/K,  \end{align} which is a model for the classifying space of $K$.   Similarly,  when thinking of $EU(\vec{k})$ as a $K$-space,  we will denote it by $EK_{\rho}.$  Finally,  we define $ BK_{\rho,n}:= EK_{\rho,n}/K$ to be the corresponding finite dimensional approximations of $BK$.  We will typically fix one such model for the classifying space when performing our constructions,  and for ease of notation,  we will drop $\rho$ subscripts from the notation when no confusion is possible. 

Now let $N$ be a finite $K$ CW-complex \cite[Chapter 2]{tDieck} and let $N_{borel,n}:= (N \times EK_n) /K$
be the finite dimensional approximations to the Borel mixing space $N_{borel} := (N \times EK) /K.  $  Let $V_{n,n+1}$ be a $K$-equivariant tubular neighborhood of $EK_n \subset EK_{n+1}$.  Then the quotient $V_{n,n+1} \times_K N$ is an open neighborhood of $N_{borel,n} \subset N_{borel,  n+1}$ which is homeomorphic to an oriented vector bundle over $N_{borel,n}.$  This means that even though the strata $N_{borel,n}$ are not manifolds,  the inclusions $N_{borel,n} \subset N_{borel,  n+1}$ are ``normally non-singular" (see \cite[Section 5.4.1]{GMII}) and this allows us to define Gysin pull-back maps on homology using excision and the Thom isomorphism in the usual way:  \begin{align*} H_*(N_{borel,{n+1}}) \to & H_*(N_{borel,{n+1}},  N_{borel,{n+1}} \setminus N_{borel,n})\\  \cong& H_*(V_{n,n+1} \times_K N, V_{n,n+1} \times_K N-0) \\ \cong& H_{*-\operatorname{rank}(V_{n,n+1})}(N_{borel,n})  \end{align*}

 The semi-infinite homology of $N$ is defined to be the limit 
\begin{align} \label{eq: inverselimit}   
  \hat{H}_{\ast}^K(N) = \varprojlim_n H_ {\operatorname{dim} BK_n+ \ast}(N_{borel,n}) 
\end{align} 
where the maps in the inverse system are defined using these Gysin pull-back maps.  Note that for any fixed degree,  the inverse limit \eqref{eq: inverselimit} stabilizes (see e.g.  \cite[page 79]{MR1786481} or \cite[page 601]{Graham}).  An important special case is that where $N$ is a compact,  $K$-oriented smooth manifold.  Here,  the Gysin pull-back maps in the inverse limit \eqref{eq: inverselimit} can be alternatively defined using Poincar\'e duality on each $N_{borel,n}$: 
\begin{align}  \label{eq: inverselimit2}  
  PD: H_ {\operatorname{dim} BK_n+\ast}(N_{borel,n})  \cong H^{\operatorname{dim}(N)-\ast}(N_{borel,n})  
\end{align} 
together with the usual pull-back on cohomology.  This last description shows that $\hat{H}_{\ast}^K(N)$ carries an equivariant fundamental class $[N]_K \in  \hat{H}_{\operatorname{dim}(N)}^K(N).$ 

The argument of \cite[Proposition 1]{MR1614555} also shows that these groups are independent of the model for the classifying space $BK_{\rho}.$ In fact,  although we will not make use of it in this text,  the above homology groups have more conceptual definitions that avoid the language of finite-dimensional approximations entirely.  We briefly describe a definition using the language of dg-local systems --- this approach is ``Koszul dual" to the one in \cite{MR3267019,  MR3615739} and makes the relation to the finite-dimensional approach more transparent.  Suppose that $(S,p)$ is a pointed connected CW-complex.  (We will systematically eliminate the base point from our notation.) A dg local-system (with ground field $\bk$) $E$ over $S$ is a module over $C_*(\Omega S,\bk)$ with its Pontryagin product.  We define functors: \begin{align} C^{\ast}(S;E) := \operatorname{RHom}_{C_{*}(\Omega S)}^{\ast}(\bk, E) \\ C_*(S;E) := \bk\otimes^{L}_{C_{*}(\Omega S)}E \end{align} 

It is well-known that these definitions extend the usual definitions of (co)homology with coefficients in a classical/discrete local system (\cite[\S 2.5]{Malm}).  A key example of a dg-local system occurs in the situation of a fiber bundle $ F \to Y \to S $.  Then the monodromy action gives the chains on the fiber $C_*(F)$, the structure of a dg-module over $C_*(\Omega S).$ We have a quasi-isomorphism: \begin{align}  C_*(Y) \cong C_*(S, \underline{C_*(F)}), \end{align} 
where $\underline{C_*(F)}$ indicates the corresponding local system.  There is also a direct generalization of Poincar\'e duality to this context.  Namely,  suppose $S$ has the homotopy type of a compact oriented manifold of dimension $n$.  Then there is a duality isomorphism (see \cite[Theorem 2.5.2]{Malm}): \begin{align} \cap [S] : C^{\ast}(S; E) \cong C_{n-*}(S; E) \end{align}

To connect this with equivariant topology,  notice that if $K$ acts on a CW complex $N$,  then $C_*(N)$ is a dg-module over $C_*(K) \cong C_*(\Omega BK).$ For any finite dimensional $K$ CW-complex $N$,  we can alternatively define\footnote{The $-\ast$ appears because we use homological grading conventions instead of the cohomological grading conventions to describe the homology theory. } \begin{align} \label{eq:lsapproach} \hat{H}^K_*(N):= H_*(C^{-*}(BK; \underline{C_*(N)})) \end{align}

 To see the relation with the finite-dimensional approach,  let us consider the restriction of our local system to some $BK_n$.  Then by Poincar\'e duality,  there are quasi-isomorphisms
\begin{align} \label{eq:finiteapproxiso} C^{-*}(BK_n; \underline{C_*(N)}) \cong C_{\operatorname{dim} BK_n +\ast}(BK_n; \underline{C_*(N)}) \cong C_{\operatorname{dim} BK_n +\ast}(N_{borel,n}) \end{align} 
identifying this cohomology over $BK_n$ with the levels of \eqref{eq: inverselimit}.  The restriction maps can be identified with those in \eqref{eq: inverselimit},  showing the equivalence of the two approaches.     

\begin{rem} From the local system perspective,  the equivariant fundamental class $[N]_K$ which arises from \eqref{eq: inverselimit2} corresponds to a fiberwise fundamental class $[N]_b$ over every point $b \in BK.$ \end{rem}

We will make use of the following Lemma: 

\begin{lem} \label{lem: freeactions}(compare \cite[Proposition 8]{MR1614555}) Suppose that K acts freely on a finite $K$- CW complex $N$,  then there is a natural isomorphism :  \begin{align} \label{eq:integrationfibers} \hat{H}_*^K(N) \cong H_{*-\operatorname{dim}(K)}(N/K) \end{align}  \end{lem} 
\begin{proof} We follow the proof of \cite[Proposition 8]{MR1614555} (which is stated for Chow groups but works for ordinary homology).  Namely,  after fixing an embedding $\rho: K \hookrightarrow U(k)$,  one can instead take the ``non-compact" Stiefel  manifold of full rank complex $k \times n$ matrices as a model for $EU(k)_n$ (see \cite[Section 3.1]{MR1614555}).  This sits inside the vector space $\mathbf{V}_n$  of all $k \times n$ matrices (viewed as a $K$-representation) and the complement $\mathbf{V}_n \setminus EU(k)_n$ has increasing co-dimension with $n$.  Working with non-compact models means that we should replace the levels of \eqref{eq: inverselimit}   with $H_{\operatorname{dim} BK_n +\ast}^{BM}(N_{borel,n})$ where $H_{\operatorname{dim} BK_n +\ast}^{BM}$ is the Borel-Moore homology and $\operatorname{dim} BK_n$ is the real dimension of this new model for $BK_n$ (compare \cite[page 79]{MR1786481} or \cite[page 601]{Graham} ; this can also be seen by retracing through \eqref{eq:finiteapproxiso}).   If we fix a degree,  then we have that for $n$ large enough,  \begin{align*}H_{\operatorname{dim} BK_n +\ast}^{BM}(N_{borel,n}) \cong H_{\operatorname{dim} BK_n +\ast}^{BM}(\mathbf{V}_n \times_K N). \end{align*} Finally,  because $K$ acts freely on $N$,  $\mathbf{V}_n \times_K N$ is an oriented vector bundle over $N/K$. So by the Thom-isomorphism,  we have that $$ H_{\operatorname{dim} BK_n +\ast}^{BM}(\mathbf{V}_n \times_K N) \cong H_{*-\operatorname{dim}(K)}(N/K).$$ This concludes the proof.      \end{proof}

\begin{rem} \label{lem: freeactions2} Let $K$ be an arbitrary compact,  connected Lie group which acts freely on a compact $K$-oriented smooth manifold $N$.  Then it follows from examining \eqref{eq: inverselimit2} that there is a natural ``integration over the fibers" isomorphism as in \eqref{eq:integrationfibers}.  It is easy to see that this produces the same isomorphism as Lemma \ref{lem: freeactions}.
  \end{rem}

\begin{example} Let us spell out some basic examples: 

 \begin{itemize} \item The basic case where $N=pt$ already exhibits some noteworthy features of these groups.  In view of\eqref{eq: inverselimit2},  we have that $\hat{H}_*^K(pt)=H^{-*}(BK).$ Notice that this is concentrated in non-positive degrees.  It also carries a natural multiplication,  unlike the usual equivariant homology $H_*(BK).$ It is also useful to note that $\hat{H}_*^K(pt)$ is much better behaved than $H_*(BK)$ as module over $H^{-*}(BK).$ For simplicity, we work over a field of characteristic zero $\mathbf{k}.$ Then  $$ H_*(BK) \cong \operatorname{Hom}_{\mathbf{k}}(H^{-*}(BK), \mathbf{k})$$ as a module over $H^{-*}(BK).$ Hence $H_*(BK)$ is a torsion module and infinitely generated,  while $\hat{H}_*^K(pt)$ is free of rank one.   
\item Let $N=K$ with $K$ acting by (left-) multiplication.  Then we have that $\hat{H}_*^K(K)=\mathbb{Z}[-\operatorname{dim}(K)]$ (so it is concentrated in homological degree $\operatorname{dim}(K)).$   
\end{itemize} 
 \end{example}

 Let us say that a $K$ CW-complex is locally finite if it has a finite number of cells in each dimension.  For general locally finite $K$ CW-complexes,  let $N_{\leq d}$ denote the union of cells of dimension less than $d.$ We set 
\begin{align} 
\hat{H}_\ast^{K}(N) := \varinjlim_d \hat{H}_{\ast}^K(N_{\leq d}) \label{eq:directlim}
\end{align} This is an invariant under $K$-equivariant homotopy equivalences by the $K$-equivariant cellular approximation theorem \cite[Theorem II.2.1]{tDieck}.  
More generally or any $K$-equivariant map $f: N \to N'$,  there is a pushforward map $$f_*: \hat{H}_{\ast}^K(N) \to  \hat{H}_{\ast}^K(N'). $$  
For a $K$-topological space $N$ that is $K$-equivariantly homotopy equivalent to a locally finite $K$ CW-complex $N_{CW}$,  $\hat{H}_{\ast}^K(N)$ is defined to be $\hat{H}_{\ast}^K(N_{CW})$.
These homology groups behave contravariantly (like equivariant cohomology) with respect to subgroups.  For any inclusion $i: H \subset K$,  we have a map $$i_{H \subset K}^*: \hat{H}_{\ast}^K(N) \to  \hat{H}_{\ast}^H(N). $$

\begin{rem} In some cases,  we will need to consider disjoint unions of locally finite $K$ CW-complexes.  We extend the theory to these cases by taking the direct sum over all connected components.   \end{rem}

\subsection{ $\hat{H}_*^{\hat{T}}(LG/T)$ and convolution product} \label{sec:conv}

We recommend \cite{OblomkovCarlsson} as a reference for the material in this section.  Let us begin by introducing some Lie theoretic notation that we will use throughout this paper: \begin{itemize} \item $G$ is a compact,  connected Lie group 
\item $T \subset G$ is a maximal torus and $\frak{t}$ its Lie algebra.  \item $N_G(T)$ is the normalizer of this maximal torus in $G$ \item $W$ is the Weyl group $N_G(T)/T$.   \item $\mathcal{X}(T)$ is the lattice of homomorphisms $S^1 \to T$ (co-character lattice)   \item $\tilde{W}$ is the affine Weyl-group $W \rtimes \mathcal{X}(T).$ \item $\hat{T}:= S^1 \times T$ and $\mathcal{X}^{\vee}(\hat{T})$ is the lattice of homomorphisms $\hat{T} \to S^1$ (character lattice).
\item Let $R:= H^{-*}(B\hat{T};\bk)$ and $F_R$ its field of fractions.
\item $LG$ is the space of smooth maps $S^1 \to G$
 \end{itemize} 

Choose a generator $y$ of $H^2(BS^1)$.  For any character $\sigma^{\vee} \in \mathcal{X^{\vee}}(\hat{T})$,  we have a map: $$B\sigma^{\vee}: B\hat{T} \to BS^1$$ giving rise to a canonical element $(\sigma^{\vee})^*(y):=(B\sigma^{\vee})^*(y) \in H^2(B\hat{T}).$  This correspondence in turn gives rise to an identification:  \begin{align} \label{eq:symid} \operatorname{Sym}(\mathcal{X}^{\vee}(\hat{T})) \cong R.  \end{align} 

It is now time to introduce one of the main examples that we wish to consider in this paper.  For our space,  we consider the smooth affine flag variety $N:=LG/T.$  This admits an action of $K:= \hat{T},$  where $T$ acts by left translation on
\begin{align} \label{eq:conjug} T\times  LG/T \to LG/T \\ t \cdot [\gamma(t)] = [t \cdot \gamma(t)].  \nonumber
\end{align} 
and $S^1$ acts by loop rotation
\begin{align} \label{eq:loopy} S^1 \times LG/T \to
  L G/T  \\ a \cdot [\gamma(t)] =
  [\gamma(t-a)].  \nonumber 
\end{align} 

Let $L_{poly}G$ denote the space of polynomial loops in $G$.  These are the loops $S^1 \to G$,  which extend to an algebraic map $\mathbb{C}^* \to G_{\mathbb{C}}$,  where $G_\mathbb{C}$ is the complexification of $G$ (\cite[3.5]{PS86} or \cite[Definition 2.1]{Atiyah-Pressley}).  The inclusion $L_{poly}G \hookrightarrow LG$ is clearly $S^1 \times T \times T$ equivariant.  Moreover,  we have the following well-known fact:
\begin{lem} The inclusion $L_{poly}G \hookrightarrow LG$ is a homotopy equivalence.   \end{lem}
\begin{proof} The main reference (\cite[Proposition 8.6.6]{PS86}) only explicitly addresses the case where $G$ is semi-simple (which is the most difficult part).  So,  we explain the easy argument which allows one to pass to the general case.  For notational simplicity,  note that it clearly suffices to prove the corresponding result on based loop spaces at the identity,  $\Omega_{poly} G \simeq  \Omega G.$  Observe first that the statement is trivial when $G=T_0$ is a torus (c.f.  the discussion just before \cite[Proposition 3.5.3]{PS86}).  Because $\Omega_{poly}$ and $\Omega $ are both compatible with products,  it follows that for any semisimple group $G'$,  $\Omega_{poly}(G' \times T_0) \simeq \Omega(G' \times T_0).$ 

Finally,  note that any compact connected Lie group $G$ is of the form $(G' \times T_0)/Z$ where $G'$ and $T_0$ are as before and $Z$ is a finite,  central subgroup (and thus contained in the maximal torus $T \subset G' \times T_0$).  By basic covering space theory,  $\pi_Z: \Omega(G' \times T_0) \to \Omega G$ maps homeomorphically onto its image and its image consists of those loops which lift to closed loops.  Moreover,  for any $z \in Z$,  let $(\Omega G)_z $ denote the subspace of based loops which lift to a path from the identity to $z.$ $(\Omega G)_z$ can be identified with $\pi_Z(\Omega (G' \times T_0))$ by choosing any $\gamma_z \in (\Omega G)_z$ and applying pointwise multiplication with $\gamma_z.$  We choose $\gamma_z$ by considering the induced covering of maximal tori $T \to T /Z$ and taking any co-character $S^1 \to T/Z$ which lifts to a path ending at $z$.  This is clearly polynomial so that the above identification respects polynomial inclusions.  As $\Omega G$ is a disjoint union of these  $(\Omega G)_z$,  this concludes the proof.   \end{proof} 
 Next,  consider the algebraic affine flag variety,  which is the quotient space $L_{poly}G/T \subset LG/T$.  The fixed points of the remaining $S^1 \times T$ action on $L_{poly}G/T$ are given by points of the form $\sigma [w]$,  where $\sigma \in \mathcal{X}(T)$ and  $w \in N_G(T)$ so $[w] \in W$ is a Weyl element.  They are thus canonically in bijection with elements of the affine Weyl group $\tilde{W}$.
  There is a cell decomposition (\cite[Section 8.7]{PS86}) \begin{align*} L_{poly}G/T = \cup_{\sigma[w] \in \tilde{W}} S_{\sigma[w]} \end{align*} where $S_{\sigma[w]}$ are finite dimensional cells corresponding to these fixed points.  The closures of $S_{\sigma[w]}$,  $\bar{S}_{\sigma[w]}$ are singular algebraic varieties (``affine Schubert varieties"),  that admit $\hat{T}$-equivariant resolutions $\mathcal{BS}_{\sigma[w]} \to \bar{S}_{\sigma [w]}$.  We have an additive isomorphism: 
\begin{align} 
\hat{H}_{*}^{\hat{T}}(LG/T) \cong \hat{H}_{*}^{\hat{T}}(L_{poly}G/T) \cong \bigoplus_{\sigma[w] \in \tilde{W}} R \cdot [\mathcal{BS}_{\sigma[w]}]_{\hat{T}} \label{eq:BottSam}
\end{align} 
where as before $[\mathcal{BS}_{\sigma[w]}]_{\hat{T}}$ denotes the $\hat{T}$-equivariant fundamental classes of these manifolds.  

We next describe a convolution product: \begin{align} \label{eq: hatmaff} \hat{m}: \hat{H}_*^ {\hat{T}}(LG/T) \otimes \hat{H}_*^{\hat{T}}(LG/T) \to \hat{H}_*^{\hat{T}}(LG/T)  \end{align} which turns $\hat{H}_*^ {\hat{T}}(LG/T)$ into an associative algebra.   To do this,  the first step is to notice that by a variant of Lemma \ref{lem: freeactions},  we have a natural identification:
\begin{align} \label{eq:anotherfree1} \hat{H}_{*+dim(T)}^{\hat{T} \times T}(LG) \cong \hat{H}_*^{\hat{T}}(LG/T). \end{align} 

Let $X$ be another space with an $\hat{T}$ action and let $LG \times_T X$ be the quotient of $LG \times X$ by the diagonal $T$-action (given by right multiplication on $LG$ and the $T$-action on $X$.)  There is an action of $\hat{T}$ on $LG \times_T X$ given by taking the diagonal $S^1$-action and allowing $T$ to act by left multiplication on the first factor.  We have that: 

\begin{lem} \label{lem:kunneth}   There is a Kunneth map: 
 \begin{align} \label{eq:kunnethR} \cQ: \hat{H}_{*}^{\hat{T}}(LG/T)\otimes_{R}  \hat{H}_{*}^{\hat{T}}(X) \to \hat{H}_{*}^{\hat{T}}(LG \times_T X) \end{align} where the $R$-module structure on $\hat{H}_{*}^{\hat{T}}(LG/T)$ is induced by \eqref{eq:anotherfree1} together with right-multiplication on $LG$.  \end{lem}   
\begin{proof}  Consider the product $LG \times X $ which admits an action by $(S^1)^2 \times T^3.$ Let $H_\Delta \subset (S^1)^2 \times T^3$ denote the subgroup,  isomorphic to $S^1 \times T^2$,  consisting of elements of the form $(a,a) \times (g_1,g_2,g_2) \in (S^1)^2 \times T^3.$ Composing the Kunneth map together with restriction to this subgroup gives a natural map: 
\begin{align} \label{eq:mult1} \hat{H}_{*+dim(T)}^{S^1\times T^2}(LG) \otimes_R \hat{H}_{*}^{\hat{T}}(X) \to \hat{H}_{*+dim(T)}^{H_{\Delta}}(LG \times X) \end{align}

$LG \times_T X$ is the quotient of $ LG \times X$ by the subgroup (isomorphic to $T$) of $T^3$ consisting of elements of the form $(1, g,g) \in T^3$.  The integration along the fibers map from Lemma \ref{lem: freeactions} gives a map: 
\begin{align} \label{eq:mult2} \hat{H}_{*+dim(T)}^{H_{\Delta}}(LG \times X)  \to \hat{H}_{*}^{\hat{T}}(LG \times_T X).  \end{align} The map  \eqref{eq:kunnethR} is the composition of these two maps after making the identification \eqref{eq:anotherfree1} in the source.  \end{proof} 

\begin{rem} Using the fact that $\hat{H}_{*}^{S^1 \times T \times T}(LG)$ is free over $R$,  one can actually argue that \eqref{eq:kunnethR} is an isomorphism.  However,  we will not make use of this stronger statement.   \end{rem}

With this in place,  note that there is an $\hat{T}$ equivariant map \begin{align}\label{eq:pointmult} m_{LG}: LG \times_T (LG/T) \to LG/T \end{align} given by pointwise multiplication in $LG$.  Composing $m_{LG,*}$ with \eqref{eq:kunnethR} for $X=LG/T$ gives the desired multiplication \eqref{eq: hatmaff}.  After tensoring with the fraction field $F_R$,  this algebra admits a simple description.  Consider the free $F_R$ module with basis indexed by $\sigma[w] \in \tilde{W}:$ \begin{align*} \mathcal{N}:= \bigoplus_{\sigma[w] \in \tilde{W}} F_R \cdot e_{\sigma[w]}.  \end{align*}


 We construct an associative algebra structure $(\mathcal{N},  \ast_{fp})$ on $\mathcal{N}$ as follows.   Consider the semi-direct product group $S^1 \ltimes LG$,  where $S^1$ acts by loop-rotation\footnote{Here this is a right action by positive rotation $\gamma(t)\cdot a=\gamma(t+a)$.} and let $N(\hat{T})$ denote the normalizer of $\hat{T} \subset S^1 \ltimes LG$.  Recall that the affine Weyl group $\tilde{W}$ admits an alternative description as the quotient $N(\hat{T})/ \hat{T}$ (\cite[Section 5.1]{PS86}).  Explicitly,  the group $\tilde{W}$ acts on the torus $\hat{T}$ by the automorphisms: 
\begin{align*}
\mathcal{A}_{\tilde{w}}: \tilde{W} \times \hat{T} &\to  \hat{T} \\
\mathcal{A}_{[w]\sigma}(a,t)&=(a,  w\sigma(a)t w^{-1})
\end{align*}
where $[w] \in W$, $\sigma \in \mathcal{X}(T)$ and $(a,t) \in S^1 \times T = \hat{T}$.  It therefore acts on $x^{\vee} \in \mathcal{X}^{\vee}(\hat{T})$ by the dual action \begin{align} \label{eq:actfun} \mathcal{A}_{\tilde{w}}(x^{\vee})(\hat{t}) = x^{\vee} \circ \mathcal{A}_{\tilde{w}^{-1}}(\hat{t}) \end{align} and thus,  in view of \eqref{eq:symid},  the rings $R$ and $F_R$.  The multiplication law on $\mathcal{N}$ is then determined by: \begin{align} \label{eq:multipbasis} (f_1e_{\tilde{w}_{1}}) \ast_{fp} (f_2e_{\tilde{w}_{2}})= f_1 (\mathcal{A}_{\tilde{w}_1}(f_2))(e_{\tilde{w}_{1}\tilde{w}_{2}}) \end{align}  where $f_1, f_2 \in F_R$ and $\tilde{w}_1:=\sigma_1[w_1], \tilde{w}_2:=\sigma_2[w_2] \in \tilde{W}$(see e.g. \cite[\S 4.1]{OblomkovCarlsson} or \cite[Chapter 4,  eq.  (3.2)]{lametal}). By a version of Atiyah-Bott localization,  the base change $F_R \otimes_{R} \hat{H}_*^{\hat{T}}(LG/T)$ admits an alternative basis: \begin{align} \label{eq:alternativebasis} F_R \otimes_{R} \hat{H}_*^{\hat{T}}(LG/T) \cong \bigoplus_{\sigma[w] \in \tilde{W}} F_R \cdot [\sigma[w]]_{\hat{T}} \end{align}
where $\sigma[w] \in LG/T$ is the fixed point and $[\sigma[w]]_{\hat{T}}$ is the fundamental equivariant class.  It is elementary to check: 

\begin{lem} There is an embedding of $R$-algebras $$ \hat{H}_{*}^{\hat{T}}(LG/T) \hookrightarrow (\mathcal{N},  \ast_{fp}) $$ which sends $[\sigma[w]]_{\hat{T}} \to e_{\sigma[w]}$ and which becomes an additive isomorphism after tensoring with $F_R.$ \end{lem}  

\begin{rem} The subalgebra $R \subset \hat{H}_{*}^{\hat{T}}(LG/T)$ is a multiplicatively-closed subset which satisfied the left (or alternatively right) Ore condition and $(\mathcal{N},  \ast_{fp})$ is the localization by this set.  However,  we shall not make use of this fact.   \end{rem} 

\section{Shift operators}\label{s:Seidelmor}

Let $(M,\omega)$ be a positively monotone compact symplectic manifold equipped with a Hamiltonian $G$-action.
Let $\Lambda=\mathbf{k}[q,q^{-1}]$ be the Laurent polynomial ring and we define the grading of $q$ to be $2$.  As a vector space over $\bk$, $QH^*_{\hat{T}}(M)$ (resp. $QH^*_{T}(M)$) is defined to be $H^*_{\hat{T}}(M) \otimes_{\bk} \Lambda$ (resp. $H^*_{T}(M) \otimes_{\bk} \Lambda$), where $S^1$ acts trivially on $M$.  Using Poincare duality \eqref{eq: inverselimit2}, we can identify it with $\hat{H}_{\dim(M)-*}^{\hat{T}}(M) \otimes_{\bk} \Lambda$ (resp. $\hat{H}_{\dim(M)-*}^{T}(M) \otimes_{\bk} \Lambda$) where $q$ is sent to $q^{-1}$.  In this and the next section, we will only use the vector space structure of $QH^{\dim(M)-*}_{\hat{T}}(M)$ (resp. $QH^{\dim(M)-*}_{T}(M)$) so  $\hat{H}_*^{\hat{T}}(M)[q^{\pm 1}]:= \hat{H}_*^{\hat{T}}(M) \otimes_{\bk} \Lambda$ (resp. $\hat{H}_*^{T}(M)[q^{\pm 1}]:= \hat{H}_*^{T}(M) \otimes_{\bk} \Lambda$) is used interchangeably with it. We denote by $u$ the $S^1$ equivariant parameter.

We are going to construct the equivariant Seidel maps
\begin{align}
&\cS: \hat{H}_*^{\hat{T}}(LG/ T) \otimes  \hat{H}_*^{\hat{T}}(M)[q^{\pm 1}] \to \hat{H}_*^{\hat{T}}(M)[q^{\pm 1}] \label{eq:Sdefn1}\\
&\cS^{u=0}: \hat{H}_*^{T}(LG/T) \otimes  \hat{H}_*^{T}(M)[q^{\pm 1}] \to \hat{H}_*^{T}(M)[q^{\pm 1}]\label{eq:Sdefn2}
\end{align}
using parametrized $2$-pointed Gromov-Witten invariants in the universal Seidel spaces.
The properties of this map will be discussed in later sections.

\subsection{Overview of the construction}\label{ss:overviewSeidelMap}

Let $D^2_0$ and $D^2_{\infty}$ be two copies of the closed unit disc, with standard coordinates \(re^{2\pi i \theta},\ \ 0\leq r\leq 1, 0\leq \theta \leq 1 \).  Let $U$ be the following space \begin{align} U:= LG \times M \times D^2_0 \bigsqcup  L G \times M \times D^2_{\infty} / \sim \label{eq:SeidelCoordinates} \end{align} where the equivalence relation is given by $(\gamma, x, e^{2\pi i \theta})_0 \sim (\gamma, \gamma(\theta)(x), e^{2\pi i \theta})_\infty$. That is, we are gluing over the boundary of \(D\), using the loop $\gamma:S^1\to G$ and the \(G\) action on $M$. Note that the fibre \(U_\gamma\) over an element \(\gamma\in LG\) recovers the original Seidel space construction associated to \(\gamma\). Thus, the projection $\pi_U: U \to L G$ makes $U$ a Seidel space
bundle over $L G$.

Define an $\hat{T}\times T = S^1 \times T \times T $ action on the charts 
\begin{align} (\tau,g,h) \cdot (\gamma, x, re^{2\pi i
\theta})_{0}&=(g\gamma(\cdot- \tau)h^{-1},  hx, re^{2\pi i
(\theta +\tau)})_{0}  \label{eq:LGactions1}\\ (\tau,g,h) \cdot (\gamma, x,
re^{2\pi i \theta})_{\infty}&=(g\gamma(\cdot- \tau)h^{-1},  gx,
re^{2\pi i (\theta +\tau)})_{\infty} \label{eq:LGactions2}\end{align}
The action is compatible with the gluing and thus it is well
defined on \(U\). To see this, note that on the \(0\)-chart,
on the boundary \(r=1\), we have 
\begin{equation*} (\tau,g, h) \cdot (\gamma, x, e^{2\pi i
  \theta})_0=(g\gamma(\cdot -\tau) h^{-1},
  hx, e^{2\pi i (\theta +\tau)})_0.
\end{equation*} After gluing the second entry in the
\(\infty\) chart is \(g\gamma(\theta+\tau -\tau)
h^{-1} hx\) and thus

\begin{equation*}
(g\gamma(\cdot -\tau)h^{-1},
g\gamma(\theta)(x), e^{2\pi i (\theta
+\tau)})_\infty=(\tau,g,h) \cdot  (\gamma, \gamma_{\theta}(x),
e^{2\pi i \theta})_\infty,
\end{equation*}
%
showing compatibility. 

Let $U/T$ be the quotient of $U$ by the last $T$-action of $\hat{T} \times T$.
It is a free quotient and we have an induced Seidel space bundle $\pi_{U/T}: U/T \to LG/T$
as well as the induced $\hat{T}$ action on $U/T$.

Consider the $\hat{T}$-invariant subspaces of $U/T$
\begin{align}
S_{T,0}&=(L G \times M \times \{0\})/T \subset (L G
\times M \times D^2_0)/T \label{eq:ST0}\\
S_{T,\infty}&=(L G \times M \times \{0\})/T \subset (L G \times M \times D^2_{\infty})/T \label{eq:STinf}
\end{align}
and let  $LG \times_T M$ denote the quotient of  $LG
\times M$ by the relation   $(\gamma,x) \sim (\gamma g^{-1},gx)$ for $g \in T$.
The subspace $S_{T,0}$ is $\hat{T}$-equivariantly isomorphic to $L G \times_T M$ with the action
\begin{align} \label{eq: Taction0}
(\tau,g) \cdot [\gamma, x]=[g\gamma(\cdot -\tau),  x]
\end{align}
On the other hand, $S_{T,\infty}$ is $\hat{T}$-equivariantly isomorphic to $(L G/T) \times M$ with the action
\begin{align*}
(\tau,g) \cdot (\gamma, x)=(g\gamma(\cdot -\tau),  gx)
\end{align*}

By counting appropriate $\hat{T}$-equivariant parametrized moduli space in $U/T$ and composing it with the map induced from the projection $(LG/T)\times M \to M$, we shall obtain a map
\begin{align}
\cC_M: \hat{H}_*^{\hat{T}}(L G \times_T M)[q^{\pm 1}] \to \hat{H}_*^{\hat{T}}((LG/T)\times M)[q^{\pm 1}] \to \hat{H}_*^{\hat{T}}(M)[q^{\pm 1}]  \label{eq:correlationmap}
\end{align}
where $\cC$ stands for correlation.  Lemma \ref{lem:kunneth} with $X=M$ gives a map:
\begin{align}  \label{eq:cQ}
\cQ: \hat{H}_{*}^{\hat{T}}(L G/T) \otimes  \hat{H}_*^{\hat{T}}(M)[q^{\pm 1}] \to \hat{H}_*^{\hat{T}}(L G  \times_T M)[q^{\pm 1}]
\end{align}
after tensoring with $\Lambda$.


The composition of $\cQ$ and $\cC_M$
will be the equivariant Seidel map $\cS$ in \eqref{eq:Sdefn1} we are going to construct.
The equivariant Seidel map in \eqref{eq:Sdefn2} is defined analogously by forgetting all the $S^1$ actions.

In the detailed construction below, we will 
\begin{itemize}
  \item work with $\hat{T}$-equivariant finite dimensional approximation of $LG/T$ instead of $LG/T$,
\item use Borel construction so we will have extra factors like $ET$, $BT$,
\item explain the choice of Floer data as well as the parametrized moduli spaces.
\end{itemize}

\subsection{Structure group of the universal bundle $U$}

For $\gamma \in L G$, we denote the fiber $\pi_U^{-1}(\gamma)$ by $X_{\gamma}$.
The space $X_{\gamma}$ is a Hamiltonian fiber bundle over
$S^2 \simeq D^2_0 \sqcup D^2_{\infty}/ \sim $ in the sense
of \cite[Chapter 6]{MSbook} meaning that the fibers of $\pi_{X_{\gamma}}: X_{\gamma} \to S^2$ are symplectic manifolds (in this case $(M,\omega)$)
and the structural group of $\pi_{X_{\gamma}}$ can be
reduced to the Hamiltonian diffeomorphism group of the
fibers (i.e. $\Ham(M,\omega)$).
A Hamiltonian bundle isomorphism
from $X_{\gamma_1}$ to $X_{\gamma_2}$ is a pair of maps $(F:X_{\gamma_1} \to X_{\gamma_2},f:S^2 \to S^2)$ such that $\pi_{X_{\gamma_2}} \circ F=f \circ \pi_{X_{\gamma_1}}$ and $F$ is a diffeomorphism
which respects structure groups.
The last condition means that there is a system of trivializations $\{U_i \times M\}_i$ of $X_{\gamma_1}$
and  $\{V_j \times M\}_j$ of $X_{\gamma_2}$ with transition functions lying in $\Ham(M,\omega)$ such that for all $i$, 
there is $j$ with  $f(U_i) \subset V_j$, and moreover, for
any $(u,m) \in U_i \times M$, there exists a smooth map
$\rho:U_i \to \Ham(M,\omega)$ such that
\begin{align*}
F:&U_i \times M \rightarrow f(U_i) \times M \subset V_j \times M\\
&(u,m) \mapsto (f(u),\rho(u)m) 
\end{align*}



We now recall the following result.

\begin{prop}[Proposition 7.1,\cite{Sav}]
The structure group of the bundle $U \to LG$ over the
connected component of $\gamma \in L G$ can be reduced
to $F^{\gamma}$, the group of Hamiltonian bundle
isomorphisms $X_{\gamma} \to X_{\gamma}$ which cover the
identity map $S^2 \to S^2$, and which agree with the identity over 
$D^2_{\infty}$ and 
over a neighbourhood of 
 $0 \in D^2_{0}$,
with respect to the coordinates \eqref{eq:SeidelCoordinates}
above. 
\footnote{We want to remark that the space $U$ in
  \cite[Proposition 7.1]{Sav} is the Seidel space bundle
over $L\Ham(M)$. The space $U$ in this paper is the
restriction of that bundle over $L G$ via the natural map $L G \to L\Ham(M)$.}
\end{prop}

To illustrate the idea of this proposition, we describe explicitly the parallel transport maps of an abstract connection on the bundle (cf. \cite[Section 7.4]{Sav}).
A path $m:[0,1] \to LG$ is {\it admissible}\footnote{Admissible paths are called smooth in \cite[Section 7.4]{Sav}} if it is locally constant near the endpoints and the associated map
$\tilde{m}: [0,1] \times S^1 \to G$ defined by $\tilde{m}(r,\theta):=m(r)(\theta)$ is smooth.
Parallel transport maps for the abstract connection are defined for admissible paths.
Let $m:[0,1] \to LG$ be an admissible path from $m(0)=\gamma_0$ to $m(1)=\gamma_1$.
The parallel transport map $F_m: X_{\gamma_0} \to X_{\gamma_1}$ is given by
\begin{align}
&(x,re^{i\theta})_0 \mapsto ((m(r)(\theta))^{-1}(m(0)(\theta))x,re^{i\theta})_0 \label{eq:parallel1}\\
&(x,re^{i\theta})_{\infty} \mapsto (x,re^{i\theta})_{\infty} \label{eq:parallel2}
\end{align}
It is readily checked to be well-defined with holonomy group $F^{\gamma}$.


We also need the corresponding result for the bundle
$U_{borel} \to (L G)_{borel}$, where the Borel
construction is taken with respect to the $S^1\times T \times \{e\}$ sub-action of the $S^1 \times T \times T$ action.
Since $S^1 \times T $ is connected, the element $(\tau,g)
\cdot \gamma$ lies in the same connected component of
$\gamma$.  Let $(L G)_{borel}^{[\gamma]}$ be the
corresponding connected component of $(L G)_{borel}$,
and $U^{[\gamma]}_{borel}$ be the corresponding connected
component of $U_{borel}$.  As a consequence of
\cite[Proposition 7.1]{Sav}, we have the following.

\begin{prop}[cf. \cite{Sav} Proposition 7.2]
The structure group of the bundle $U^{[\gamma]}_{borel} \to (LG)_{borel}^{[\gamma]}$ can be reduced to $F^{\gamma}_{borel}$, the  group of Hamiltonian bundle isomorphisms $X_{\gamma} \to X_{\gamma}$, which sit over rotation of $S^2$, with the axis of rotation corresponding to $0 \in D^2_0$ and $0 \in D^2_{\infty}$. 
Moreover,  it can be given an abstract connection such that   elements of the holonomy group of $X_{\gamma}$ act as identity on $M \times \{0\} \subset M \times D^2_0$ and act as $g \times e^{2\pi i \tau}$ on $M \times D^2_{\infty} \subset X_{\gamma'}$ for some $(\tau,g) \in S^1 \times T $.
\end{prop}

\begin{proof}
It follows from the same line of reasoning as \cite[Section 7.5]{Sav}.
The key point is that the abstract connection of the bundle $U \to L G$ induces an abstract connection on 
$U^{[\gamma]}_{borel} \to (L G)_{borel}^{[\gamma]}$ whose parallel transport map is given by  \eqref{eq:parallel1}, \eqref{eq:parallel2} up to the action by $S^1 \times T$.
\end{proof}

The isomorphism type of the group $F^{\gamma}$ (resp. $F^{\gamma}_{borel}$) is independent of $\gamma$ and we denote it by $F$ (resp. $F_{S^1 \times T}$). 

\begin{defn}
  A fiber bundle $p:P \to B$ is an $\cF$-fibration if each
  fiber \(X_b:=p^{-1}(b)\) is equipped with a Hamiltonian bundle structure 
  $\pi:X_b \to S^2$ and the structure group of $p$ is
  isomorphic to either $F$ or $F_{S^1 \times T}$.
\end{defn}

\begin{lem}\label{l:actSecTrivial}
Let $p:P \to B$ be an $\cF$-fibration, $b \in B$ and $X_b$ be the fiber over $b$.
The holonomy group of $p:P \to B$ at $X_b$ acts trivially 
on the section classes $H_2^{sec}(X_b)$ 
(i.e. the subspace of $H_2(X_b)$ spanned by the image of sections).
\end{lem}

\begin{proof}
The structure group of an $\cF$-fibration is connected so the homological action of the holonomy group is trivial.
\end{proof}


\subsection{Admissible structures}\label{ss:admissible structure}
In order to do Floer theory on these Seidel space bundles we will need to introduce some auxiliary structures.  We will follow the classical construction of symplectic structures on fibrations as in \cite{gls-fib,MSbook, Sav}.  

\begin{defn}
Let $\pi_X:X \to S^2$ be a smooth Hamiltonian bundle and $\{\omega_a\}_{a \in S^2}$ be its associated fiberwise symplectic forms. A symplectic form $\Omega$ on $X$ is called $\omega$-compatible if $\Omega|_{\pi_X^{-1}(a)}=\omega_a$ for every $a \in S^2$.
\end{defn}
The space of $\omega$-compatible symplectic forms is non-empty and weakly contractible (\cite[Lemma 2.5]{Sav}).\footnote{It is denoted by $\mathcal{A}$ in the proof of the result.}

\begin{example}[\cite{Sav}, before Definition 2.6]\label{e:SympForm}
Let $\gamma \in LG$. An  $\omega$-compatible symplectic form $\widetilde{\Omega}_{\gamma}$ on $X_{\gamma}$ can be explicitly constructed as follows.
Let $\epsilon >0$ and $0< \delta < \frac{1}{2}$.
Let $\eta(r):[0,1] \to [0,1]$ be an increasing smooth function such that $\eta(r)=1$ when $1-\delta \le r \le 1$ and $\eta(r)=r^2$ when $r \le 1-2\delta$.

Over $M \times D^2_0$, 
\begin{align*}
\widetilde{\Omega}_{\gamma}:= \omega + \epsilon 2r dr \wedge d \theta
\end{align*}
Over $M \times D^2_{\infty}$, 
\begin{align*}
\widetilde{\Omega}_{\gamma}:=& \omega + d(\eta(r) H^{\gamma}_{\theta}(\gamma(0)^{-1}x)) \wedge d\theta
-\max_x H^{\gamma}_{\theta}(x) d\eta \wedge d\theta -\epsilon 2r dr \wedge d \theta 
\end{align*}
where $(H^{\gamma}_{\theta})_{\theta \in S^1}$ is the normalized generating Hamiltonian function of $\gamma^{-1}(0) \circ \gamma$, and $D^2_{\infty}$ is given the negative of the standard orientation.

We can write $\widetilde{\Omega}_{\gamma}=\omega+\beta_1+\beta_2$ where 
\begin{align}
\beta_1&=\eta(r) dH^{\gamma}_{\theta}(\gamma(0)^{-1}x) \wedge d\theta \label{eq:beta1}\\
\beta_2&=(H^{\gamma}_{\theta}(\gamma(0)^{-1}x)-\max_x H^{\gamma}_{\theta}(x)) d\eta \wedge d\theta -\epsilon 2r dr \wedge d \theta \label{eq:beta2}
\end{align}
The fact that $\iota_v\beta_2=0$ for all $v \in TM$, $\beta_2|_{TD^2_{\infty}}>0$, $\iota_{\partial_r}\beta_1=0$ and $\beta_1|_{TM}=0$ guarantees that 
$\widetilde{\Omega}_{\gamma}$ is non-degenerate.

To see why, let $v=v_M+v_D \in TM \oplus TD^2_{\infty}$. If $\beta_2(v_D,\partial_r) \neq 0$, then $\widetilde{\Omega}_{\gamma}(v,\partial_r)= \beta_2 (v_D,\partial_r)  \neq 0$.
If $\beta_2(v_D,\partial_r) =0$, then $v_D$ is a multiple of $\partial_r$ and there are two cases depending on whether $v_M=0$.
If $v_M=0$, then $\widetilde{\Omega}_{\gamma}(v,\partial_{\theta})=\beta_2(v_D,\partial_{\theta})\neq 0$.
If $v_M \neq 0$, then we can find another vector $v_M'$ in $TM$ such that $\widetilde{\Omega}_{\gamma}(v,v_M')=\omega(v_M,v_M') \neq 0$.
Therefore, $\widetilde{\Omega}_{\gamma}$ is non-degenerate.
\end{example}

\begin{defn}
For an $\cF$-fibration $p:P \to B$, an admissible family of symplectic structures $\{\Omega_b\}_{b \in B}$
is a family of $\omega$-compatible structures.
When $B$ is a finite dimensional smooth manifold (possibly with boundary) and $p$ is smooth, we require that $\{\Omega_b\}_{b \in B}$ is a smooth family.
\end{defn}

\begin{rem}
More precisely, in order to make sense of $\{\Omega_b\}_{b \in B}$, we observe that the space of $\omega$-compatible symplectic forms is acted by the structural group $\mathcal{F}$.
The admissible family of symplectic structures $\{\Omega_b\}_{b \in B}$ shall be defined as a section from $B$ to the associated bundle with fibers being the space of  $\omega$-compatible symplectic forms.
\end{rem}

If $B$ has the homotopy type of a CW complex, then by obstruction theory, the space of admissible families of symplectic structures is
not empty and weakly contractible because the space of $\omega$-compatible symplectic forms is.
In particular, it applies when $B=LG$ or $B=LG_{borel}$.

\begin{example}\label{e:admissSym}
Recall $\widetilde{\Omega}_{\gamma}$ from Example \ref{e:SympForm}.
The family $\{\widetilde{\Omega}_{\gamma}\}_{\gamma \in LG}$ forms an admissible family of symplectic structures
for $U \to LG$.
By averaging over the $S^1 \times T \times T$-action, we can get an $S^1 \times T \times T$-invariant admissible family of symplectic structures $\{\widetilde{\Omega}_{\gamma}^{aver}\}_{\gamma \in LG}$ for $U \to LG$.
Over $\{\gamma\} \times M \times D^2_{\infty}$, $\widetilde{\Omega}^{aver}_{\gamma}$ is given by $\omega+\beta_1^{aver}+\beta_2^{aver}$ where $\beta_i^{aver}=\int_{S^1 \times T \times T} (\tau,g,h)^* \beta_i d(\tau,g,h)$ is the average of $\beta_i$.
The conditions $\iota_v\beta_2^{aver}=0$ for all $v \in TM$, $\beta_2^{aver}|_{TD^2_{\infty}}>0$, $\iota_{\partial_r}\beta_1^{aver}=0$ and $\beta_1^{aver}|_{TM}=0$
remain valid so $\widetilde{\Omega}^{aver}_{\gamma}$ is non-degenerate.

Using the projection $LG \times S^{\infty} \times ET\to LG$, we can pull back  $\{\widetilde{\Omega}^{aver}_{\gamma}\}_{\gamma \in LG}$ to get an admissible family of $S^1 \times T \times T$-invariant symplectic structures for $U \times S^{\infty} \times ET\to LG \times S^{\infty} \times ET$. It descends to an admissible family of $T$-invariant symplectic structures $\{\widetilde{\Omega}^{borel}_{\gamma}\}_{\gamma \in LG_{borel}}$
for $U_{borel} \to (LG)_{borel}$.


\end{example}





\begin{defn}
For a smooth Hamiltonian bundle $\pi_X:X \to S^2$ with an
$\omega$-compatible symplectic form $\Omega$, an almost
complex structure $J$ is called $\pi$-compatible if $\pi_X$
is $J$-holomorphic, $J$ is $\Omega$-tamed and the horizontal
tangent space $(T\pi_X^{-1}(s))^{\Omega}$ is preserved by
$J$.  
\end{defn}

The space of $\pi$-compatible  almost
complex structures is weakly contractible \cite[Proposition
2.5.13]{MSbook}, \cite[Lemma 8.2.8]{MSJbook}.  

\begin{defn}
For a smooth finite dimensional manifold $B$ (possibly with boundary) 
and a smooth
$\cF$-fibration $p:P \to B$ with an admissible family of
symplectic structures $\{\Omega_b\}_{b \in B}$, we call a
smooth family of almost complex structures $\{J_b\}_{b \in
  B}$  $\pi$-compatible if $J_b$ is $\pi$-compatible for all
  $b \in B$.
\end{defn} 

Again, it follows that the space of $\pi$-compatible
family of almost complex structures $\{J_b\}_{b \in B}$ is
weakly contractible. 

\subsection{Equivariant cycles}\label{ss:equiCycle}

Let $B$ be an oriented smooth manifold (possibly with boundary) and $f:B \to  (LG)_{borel}$ be a continuous map.
We have the following pull-back diagram such that $Y$ is a principal $\hat{T}$ bundle over $B$ and $F$ is $\hat{T}$ equivariant
\begin{equation}\label{eq:Fmap}
\xymatrix{
Y \ar[rrrr]^{F=(F_H,F_S,F_{T})} \ar[d] & & & & LG \times S^{\infty} \times ET   \ar[d]\\
B \ar[rrrr]_{f} &&  & &(LG)_{borel}
}
\end{equation}


\begin{defn}\label{d:smooth}
We call $F_H$ {\it smooth} if the associated map $\tilde{F}_H:Y \times S^1 \to G$ defined by $\tilde{F}_H(y,\theta):=F_H(y)(\theta)$ is smooth.
We call $f$ {\it smooth} if $F_H$, $F_S$ and $F_{T}$ are smooth.  

\end{defn}

We can also pull back the $\cF$-fibration
\begin{equation}
\xymatrix{
P_f \ar[rr] \ar[d]_{p_f}  & &U_{borel}  \ar[d]\\
B \ar[rr]_{f} & &(LG)_{borel}.
}
\end{equation}
When $f$ is smooth, $P_f$ and $p_f$ are smooth. 
Moreover, we can pull back the admissible family of symplectic structures $\{\widetilde{\Omega}^{borel}_{\gamma}\}_{\gamma \in LG_{borel}}$ in Example \ref{e:admissSym}
to get a family of symplectic structures for $p_f:P_f \to B$.
We can perturb this family  of symplectic structures to make it smooth and hence admissible.\footnote{Before perturbation, it might not be smooth because the construction of $\widetilde{\Omega}_{\gamma}$ uses $\max$ and it might not vary smoothly as we vary $\gamma$. See also the proof of Lemma \ref{r:inverselimit}.}


The space $P_f$ can be described as follows.
Let
\begin{align}
P_{F}:=Y \times M \times D^2_0 \sqcup Y \times M \times D^2_{\infty} / \sim \label{eq:Seidelbun}
\end{align}
where the equivalent relation is given by $(y, x, e^{2\pi i \theta})_0 \sim (y, F_H(y)(\theta)(x), e^{2\pi i \theta})_\infty$.
It has an $\hat{T}$ action as before
\begin{align*}
(\tau,g) \cdot (y, x, re^{2\pi i \theta})_{0}&=((\tau,g) \cdot y,  x, re^{2\pi i (\theta +\tau)})_{0}\\
(\tau,g) \cdot (y, x, re^{2\pi i \theta})_{\infty}&=((\tau,g) \cdot y,  gx, re^{2\pi i (\theta +\tau)})_{\infty}
\end{align*}
The quotient of $P_{F}$ by $\hat{T}$ is $P_f$.

We also have a $\mathbb{P}^1$-bundle over $B$ given by forgetting the $M$ factor in \eqref{eq:Seidelbun}
\begin{align} 
((Y \times D^2_0 \sqcup Y \times D^2_{\infty})/\sim)/\hat{T} \to B \label{eq:Spherebun}
\end{align}
We denote the fiber over $b \in B$ by $\mathbb{P}^1_b$.  The fibration $P_f$ contains two distinguised submanifolds which will be important in the sequel.  Let $\tilde{i}_{0,\infty}: Y \times M \hookrightarrow Y \times M \times D^2_{0,\infty} \subset P_{F}$ be the embeddings over the origin in $D^2_{0}$ and $D^2_{\infty}$, respectively. The image of $\tilde{i}_{0,\infty}$ is $\hat{T}$-invariant so it inherits a  $\hat{T}$ action.
Moreover, it descends to \begin{align} \label{eq: divisors} i_{0,\infty}: (P_f)_{0,\infty}:=(Y \times M)/\hat{T} \hookrightarrow P_f. \end{align}


\subsection{Main construction}\label{ss:parametrizedmoduli}


Our actual constructions will take place over $(LG/T)_{borel}$,  so we must consider cycles in this space.   Let $\bar{B}$ be an oriented smooth manifold (possibly with boundary) and $\bar{f}:\bar{B} \to  (LG/T)_{borel}$ be a continuous map.
We have the following pull-back diagram such that $B$ is a principal $T$-bundle over $\bar{B}$ and $f$ is $T$- equivariant (with respect to right multiplication):
\begin{equation}\label{eq:liftdiag}
\xymatrix{
B \ar[rrrr]^{f} \ar[d] & & & & (LG)_{borel}   \ar[d]\\
\bar{B} \ar[rrrr]_{\bar{f}} &&  & &(LG/T)_{borel}
}
\end{equation}

\begin{defn} \label{d:smoothbarB}
We say that $\bar{f}:\bar{B} \to  (LG/T)_{borel}$  is smooth if $f: B \to LG_{borel}$ is smooth in the sense of Definition \ref{d:smooth}.  \end{defn}

\begin{lem}\label{l:smoothApproximation}
Any continuous map $\bar{f}:\bar{B} \to  (LG/T)_{borel}$ is the $C^0$-limit of a sequence of smooth maps
$(\bar{f}_k:\bar{B} \to  (LG/T)_{borel})_{k \in \mathbb{N}}$.
Moreover, if the composition of $\bar{f}$ and the projection $\Pi:(LG/T)_{borel} \to B\hat{T}$ is smooth, then 
$\bar{f}_k$ can be chosen such that $\Pi \circ \bar{f}_k$ is independent of $k$.

Any continuous map $\bar{c}:[0,1] \times \bar{B} \to  (LG/T)_{borel}$ is the $C^0$-limit of a sequence of smooth maps
$(\bar{c}_k: [0,1] \times \bar{B} \to  (LG/T)_{borel})_{k \in \mathbb{N}}$.
Moreover, if $\bar{c}(0,-)$ and $\bar{c}(1,-)$ are smooth and $\bar{c}(t,-)$ is indepndent of $t$ near $0$ and $1$ respectively, then we can assume that $\bar{c}_k(0,-)=\bar{c}(0,-)$
and $\bar{c}_k(1,-)=\bar{c}(1,-)$ for all $k$.
\end{lem}

\begin{proof}
We start with the first part of the lemma.
Let $(F_H,F_S,F_T): Y \to LG \times S^{\infty} \times ET$ be the $S^1 \times T\times T$-equivariant map associated to $\bar{f}:\bar{B} \to  (LG/T)_{borel}$ as above.
It suffices to show that $\tilde{F}_H: Y \times S^1 \to G$, $F_S$ and $F_T$
can be  approximated by $S^1 \times T \times T$-equivariant, $S^1$-equivariant and $T$-equivariant smooth maps respectively.
These approximations exist by a general equivariant smoothing theorem (\cite[Theorem 4.2 on p.317]{Bredon72}).
If $\Pi \circ \bar{f}$ is already smooth to begin with, then $F_S$ and $F_T$ are smooth.
Therefore, we only need to choose an equivariant smooth approximation of $\tilde{F}_H$
and the corresponding $\bar{f}_k$ would have $\Pi \circ \bar{f}=\Pi \circ \bar{f}_k$ for all $k$.
This finishes the proof of the first part.

The proof of the second part is similar.  Since  $Y$ is a principal $\hat{T} \times T$ principal bundle, there is a simple proof for the existence of the equivariant smoothing which does not make use of the general theorem.
We provide the arguments below for readers convenience.

Since $Y$ is compact, there exists $m \in \mathbb{N}$ such that $F_S(Y) \subset S^{2m+1}$
and $F_T(Y) \subset (ET)_m$.
Therefore, we shall assume that $F_S:Y \to S^{2m+1}$ and $F_T:Y \to (ET)_m$.
Note that  both domain and codomain are finite dimensional, and the actions are free and smooth.
Therefore, a smooth approximation of the map between quotients lifts to an $S^1$-equivariant (resp. $T$-equivariant) smooth approximation of $F_S$ (resp. $F_T$).

Finally, we consider $\tilde{F}_H$. The $\hat{T} \times T$-equivariancy means that
$\tilde{F}_H((\tau,g,h)y,\theta)=g\tilde{F}_H(y,\theta -\tau)h^{-1}$.
Therefore, $\tilde{F}_H^{res}:=\tilde{F}_H|_{Y \times \{e\}}$ determines $\tilde{F}_H$ completely, and the smoothness of $\tilde{F}_H^{res}$ is equivalent to the smoothness of $\tilde{F}_H$.
As a result, it suffices to find a $T \times T$-equivariant smooth approximation of $\tilde{F}_H^{res}$.

Since $Y$ is a principal $\hat{T} \times T$-bundle (in particular, a principal $T \times T$-bundle), there exists a $T \times T$-equivariant map $F_{T\times T}:Y \to E(T \times T)_r$ for some $r \in \mathbb{N}$. 
The map $(\tilde{F}_H^{res},F_{T \times T}): Y \to G \times E(T \times T)_r$ is  $T \times T$-equivariant so it descends to a map
$Y/(T \times T) \to (G \times E(T \times T)_r)/T \times T$.
A smooth approximation of this descended map lifts to a $T \times T$-equivariant smooth approximation of $(\tilde{F}_H^{res},F_{T \times T})$.
In particular, the first component would be a $T \times T$-equivariant smooth approximation of $\tilde{F}_H^{res}$.
This completes the proof for the first part where $\Pi \circ \bar{f}$ is not assumed to be smooth.

\end{proof}

Given a smooth map $\bar{f}:\bar{B} \to  (LG/T)_{borel}$,  the induced $\cF$-fibration $p:P_f \to B$ inherits a free $T$-action.  By quotienting by this action, we obtain a fibration $\bar{p}: \bar{P}_f \to \bar{B}$ which we refer to as the T-reduced $\cF$-fibration associated to $\bar{f}:\bar{B} \to  (LG/T)_{borel}$.  We can equip this $T$-reduced $\cF$-fibration with families of symplectic forms as follows: On $p:P_f \to B$,  we can restrict to those admissible families of symplectic structures $\{\Omega_b\}_{b \in B}$ which are $T$-invariant (cf. Example \ref{e:admissSym}).   
By averaging over the group,  we again see that this is a non-empty and connected space.\footnote{We might need to add a positive multiple of the area form from the base of the fibration $\pi_b: X_b \to \mathbb{P}^1_b$ to guarantee that after taking average, we get a family of non-degenerate $2$-forms.}  
Given a $T$-invariant admissible family of symplectic structures $\{\Omega_b\}_{b \in B}$,  we can view this structure as being pulled-back from a family  $\{\Omega_{\bar{b}}\}_{\bar{b} \in \bar{B}}$ for $\bar{p}: \bar{P}_f \to \bar{B}$.  

Following the overview in Section \ref{ss:overviewSeidelMap}, we have to consider the corresponding cycles mapping to $(S_{T,0})_{borel}, (S_{T,\infty})_{borel} \subset U_{borel}/T$ (cf. \eqref{eq:ST0}, \eqref{eq:STinf}).  These will correspond to the submanifolds $(\bar{P}_f)_{0,\infty}:=(P_f)_{0,\infty}/T$ where $(P_f)_{0,\infty}$ is defined in \eqref{eq: divisors}.
Concretely, we have (cf. \eqref{eq:LGactions1}, \eqref{eq:LGactions2}) 
\begin{align}
(\bar{P}_f)_{0}=Y \times M/ (y, x) \sim ((\tau,g,h) \cdot y,  hx)\\
(\bar{P}_f)_{\infty}=Y \times M/ (y, x) \sim ((\tau,g,h) \cdot y,  gx) \label{eq:PtoSinf}
\end{align}
so there are natural maps $(\bar{P}_f)_{0/\infty} \to (S_{T,0/\infty})_{borel}$.

Let $\bar{B}$ be a {\it closed} oriented manifold.
For a smooth map $\bar{f}:\bar{B} \to (LG/T)_{borel}$, we will first use parametrized moduli spaces in $\bar{P}_f$ to define a map
\begin{align*}
\cC_{\bar{f}}: H_*((\bar{P}_f)_0)[q^{\pm 1}] \to H_*((\bar{P}_f)_{\infty})[q^{\pm 1}]
\end{align*}
Then we will show that these $\cC_{\bar{f}}$ are compatible for various $\bar{f}$ and they form the finite dimensional analogue of the map \eqref{eq:correlationmap}
 in Section \ref{ss:overviewSeidelMap}. To define these parameterized moduli spaces,  choose a family $\{\Omega_{\bar{b}}\}_{\bar{b} \in \bar{B}}$ of  symplectic structures on $\bar{p}_f:\bar{P}_f \to \bar{B}$ as above.  Let $\{J_{\bar{b}}\}_{\bar{b} \in \bar{B}}$ be a $\pi$-compatible family of almost complex structures with respect to $\{\Omega_{\bar{b}}\}_{\bar{b} \in \bar{B}}$.  

\begin{defn}

Let $\cM_{0,2}(\bar{P}_f, A, \{J_{\bar{b}}\})$ be the parametrized moduli space which consists of pairs $(u,\bar{b})$ such that
\begin{itemize}
\item $\bar{b} \in \bar{B}$
\item $u:\mathbb{P}^1_{\bar{b}} \to X_{\bar{b}}$ is $J_{\bar{b}}$-holomorphic section representing the class $A \in \im(H_2^{sec}(X_{\bar{b}}) \to H_2(\bar{P}_f))$
\end{itemize}
where $\mathbb{P}^1_{\bar{b}}$ is the fiber over $\bar{b}$ of the $\mathbb{P}^1$-bundle defined by quotienting \eqref{eq:Spherebun} by the remaining $T$-action, and $X_{\bar{b}}$ is the Seidel space associated to $\bar{b}$.
\end{defn}
The class $A$ is unambiguous because the structural group of an $\cF$-fibration acts trivially on the section classes (Lemma \ref{l:actSecTrivial}), and the same is true for the $T$-reduced $\cF$-fibration.  

The virtual dimension of $\cM_{0,2}(\bar{P}_f, A, \{J_{\bar{b}}\})$
is given by
\begin{align}
\vdim(\cM_{0,2}(\bar{P}_f, A, \{J_{\bar{b}}\}))=\dim(\bar{B})+2n+2c_1^{vert}(A) \label{eq:vdim}
\end{align}

\begin{lem}\label{l:Gcomp}
  For generic $\pi$-compatible family of almost complex structures $\{J_{\bar{b}}\}$, the moduli space $\cM_{0,2}(\bar{P}_f, A, \{J_{\bar{b}}\})$ is regular and admits a Gromov compactification $\ol{\cM}_{0,2}(\bar{P}_f, A, \{J_{\bar{b}}\})$ by adding strata of real codimension at least $2$.
\end{lem}

\begin{proof}
Given  any $\bar{b} \in \bar{B}$,  there is a small neighborhood $U_{\bar{b}}$ over which we can identify the fibers of $\bar{P}_f \to \bar{B}$ by a choice of local trivialization. 
Then by the smoothness of $\{J_{\bar{b}}\}$ and $\{\Omega_{\bar{b}}\}$,  as well as the energy uniform bound in this neighborhood,  we can apply Gromov compactness because we can regard all curves in fibers over $U_{\bar{b}}$ as lying in $X_{\bar{b}}$ using the local trivialization.  To verify the claim about the additional strata having real codimension at least two,  note that a point in the Gromov compactification is a genus $0$ stable maps to $X_{\bar{b}}$.
Since $A$ is a section class,  every component of the stable map except the main component maps into a fiber of $\pi_{\bar{b}}:X_{\bar{b}} \to \mathbb{P}^1_{\bar{b}}$.
The fiber is symplectomorphic to $(M,\omega)$ so the monotonicity of $M$ guarantees that these stable maps form lower dimensional strata of real codimension at least $2$ (cf. \cite[Chapter 6 and 8]{MSJbook}).
\end{proof}

We denote the point in $\mathbb{P}^1_{\bar{b}}$ coming from the origin of $D^2_0$ be $z_0$, and the one coming from the origin of $D^2_{\infty}$ be  $z_{\infty}$.
We denote the corresponding evaluation maps $\ol{\cM_{0,2}}(\bar{P}_f, A, \{J_{\bar{b}}\}) \to \bar{P}_f$ by $\ev_0$ and $\ev_{\infty}$, respectively.
By Lemma \ref{l:Gcomp}, the evaluation maps define pseudo-cycles on $\bar{P}_f$.

We define the following linear map
\begin{align}
\cC_{\bar{f}}: &H_*((\bar{P}_f)_0)[q^{\pm 1}] \to H_*((\bar{P}_f)_{\infty})[q^{\pm 1}] \label{eq:correlationf}\\
& x \mapsto  \sum_{A\in \im(H_2^{sec}(X_{\bar{b}}) \to H_2(\bar{P}_f))} [\ev_{\infty}: \ol{\cM_{0,2}}(\bar{P}_f, A, \{J_b\}) \times_{\ev_0} Z_x \to (\bar{P}_f)_{\infty}] q^{-c_1^{vert}(A)}
\end{align}
where $\times_{\ev_0} Z_x$ stands for fiber product between $\ev_0$ and a pseudocycle $Z_x$ in $(\bar{P}_f)_0$ representing $x$. 
Because of \eqref{eq:vdim} and $\deg(q)=2$, $\cC_{\bar{f}}$ is degree-preserving.

\begin{lem}\label{l:correlationmap}
The map $\cC_{\bar{f}}$ is well-defined.  In other words, it is independent of the choice of $\{J_{\bar{b}}\}$, $\{\Omega_{\bar{b}}\}$ and the choice of pseudocycle representing $x$ as long as they are generic.
\end{lem}

\begin{proof}

  For a fixed $A$, a standard cobordism argument shows that the pseudo-cycles defined by $\ev_{0/\infty}: \ol{\cM}_{0,2}(\bar{P}_f, A, \{J_{\bar{b}}\}) \to \bar{P}_f$ is independent of the choice of a regular family $\{J_{\bar{b}}\}$ (cf. \cite[Chapter 6 and 8]{MSJbook}) because $(M,\omega)$ is monotone.

Since the space of $T$-invariant admissible family of symplectic forms $\{\Omega_{b}\}_{b \in B}$ is connected, the space of family of symplectic form $\{\Omega_{\bar{b}}\}_{\bar{b} \in \bar{B}}$ descended from $T$-invariant admissible family is also connected.
Therefore, we can apply a further cobordism argument to show that  the pseudo-cycles are also 
 independent of the choice of $\{\Omega_{\bar{b}}\}_{\bar{b} \in \bar{B}}$  (cf. \cite[Remark 7.1.11]{MSJbook}, \cite[Lemma 3.8]{Sav}).

 A further cobordism argument shows that the $[\ev_{\infty}: \ol{\cM}_{0,2}(\bar{P}_f, A, \{J_b\}) \times_{\ev_0} Z_x \to (\bar{P}_f)_{\infty}] $ is also independent of the choice of a pseudocycle $Z_x$ representing $x$.

Finally, we need to argue that the sum in \eqref{eq:correlationf} is finite.
Let $\mathcal{A}_k$ be the collection of classes in
$\im(H_2^{sec}(X_{\bar{b}}) \to H_2(\bar{P}_f))$ such that $c_1^{vert}=k$.
For any two classes $A,A' \in \mathcal{A}_k$, there difference is a fiber class $F$ with $c_1(F)=0$.
By monotonicity of $(M,\omega)$, $\omega(F)=0$ and hence $\Omega_{\bar{b}}(A)=\Omega_{\bar{b}}(A')$.
By Gromov compactness, the union 
$\cup_{A \in \mathcal{A}_k} \ol{\cM}_{0,2}(\bar{P}_f, A, \{J_{\bar{b}}\})$ is compact.
 By the dimension formula \eqref{eq:vdim}, the pseudo-cycle $[\ev_{\infty}: \ol{\cM}_{0,2}(\bar{P}_f, A, \{J_b\}) \times_{\ev_0} Z_x \to (\bar{P}_f)_{\infty}] $ has dimension
\begin{align*}
2c_1^{vert}(A)+\deg(x)
\end{align*}
which is lying in between $0$ and $\dim((\bar{P}_f)_{\infty})=\dim(\bar{B})+2n$ only if
\begin{align}
-\deg(x) \le 2c_1^{vert}(A) \le \dim(\bar{B})+2n-\deg(x). \label{eq:possibleContribution}
\end{align}
Therefore, only finitely many $c_1^{vert}(A)$ can contribute to $\cC_{\bar{f}}(x)$.


\end{proof}

Later on, when we take inverse limit in the Borel direction, we will consider a sequence of $\bar{B}$ with increasing dimension.
Therefore, the number of values of $c_1^{vert}(A)$ satisfying equation \eqref{eq:possibleContribution} will diverge to infinity when the dimension of $\bar{B}$ increases to infinity.
A priori,  this would be a problem and force us to work with a completion of $\Lambda=\mathbf{k}[q^{\pm 1}]$.
However, since we are going to work with the semi-infinite homology, 
the number $d(x):=\dim(\bar{B})+2n-\deg(x)$ will be fixed when we take inverse limit in the Borel direction.
It means that \eqref{eq:possibleContribution} becomes 
\begin{align*}
-\dim(\bar{B})-2n+d(x) \le 2c_1^{vert}(A) \le d(x).
\end{align*}
and,  a priori, we have an upper bound of $c_1^{vert}(A)$ but not a lower bound when $\dim(\bar{B})$ goes to infinity.
However, if we choose $\{\Omega_{\bar{b}}\}_{\bar{b} \in \bar{B}}$ more carefully,  there is a uniform upper bound of $\Omega_{\bar{b}}(A)$ depending only on $c_1^{vert}(A)$.
Moreover, for sufficiently negative $c_1^{vert}(A)$, we have $\Omega_{\bar{b}}(A)<0$ and hence the moduli $\ol{\cM}_{0,2}(\bar{P}_f, A, \{J_{\bar{b}}\})$ will be empty.  This allows us to work over $\Lambda=\bk[q^{\pm 1}]$ as opposed to its completion.  The precise statement is the following lemma.

\begin{lem}\label{r:inverselimit}
Let $K \subset LG$ be a compact set that is $S^1 \times T \times T$-invariant.
There is a positive number $N$ depending only on $K$ such that for any smooth map
$\bar{f}:\bar{B} \to  (LG/T)_{borel}$ with image $im(\bar{f}) \subset (K/T)_{borel}$,
the moduli space $\ol{\cM}_{0,2}(\bar{P}_f, A, \{J_b\})$ is null-cobordant
if $c_1^{vert}(A) < -N$.
\end{lem}

\begin{proof}
Let $A_0 \in \im(H_2^{sec}(X)\to H_2(\bar{P}_f))$ (regarded as a reference section class).
Recall $\widetilde{\Omega}^{aver}_{\gamma}$ from Example \ref{e:admissSym}.
Since $K$ is compact, there is $U \in \mathbb{R}$ 
such that for all $\gamma \in K$, we have $\widetilde{\Omega}^{aver}_{\gamma}(A_0) \le U$. 
Therefore, we have $\widetilde{\Omega}^{borel}_{\gamma}(A_0) \le U$ for any $\gamma \in K_{borel} \subset (LG)_{borel}$.

Let $N=\frac{U}{\lambda}-c_1^{vert}(A_0) +1$ where $\lambda$ is the monotonicity constant such that $[\omega]=\lambda c_1(M)$.

Let $\bar{f}:\bar{B} \to  (LG/T)_{borel}$ be a smooth map that lies in $(K/T)_{borel}$.
If $\{\Omega_{\bar{b}}'\}_{\bar{b} \in \bar{B}}$ is the pull-back of $\{\widetilde{\Omega}^{borel}_{\gamma}\}_{\gamma}$, then $\Omega_{\bar{b}}'(A_0) \le U$.
Any section class $A$ can be uniquely written as $A=A_0+A_F$ where $A_F \in H_2(M,\mathbb{Z})$.
Therefore, we have 
\begin{align*}
\Omega_{\bar{b}}'(A) \le U+\omega (A_F)=U+\lambda c_1(A_F)=U+\lambda (c_1^{vert}(A)-c_1^{vert}(A_0))
\end{align*}
so $\Omega_{\bar{b}}'(A)<0$ when $c_1^{vert}(A)<-N<-(\frac{U}{\lambda}-c_1^{vert}(A_0))$.

However, $\{\Omega_{\bar{b}}'\}_{\bar{b} \in \bar{B}}$ is not necessarily smooth with respect to $\bar{b}$ so we need to perturb it to make it smooth.
We claim that after a sufficiently small perturbation, we have an admissible $\{\Omega_{\bar{b}}\}_{\bar{b} \in \bar{B}}$
such that $\Omega_{\bar{b}}(A_0) \le U+\lambda$.
Therefore, $\Omega_{\bar{b}}(A)$ is bounded above by $U+\lambda (c_1^{vert}(A)-c_1^{vert}(A_0))+\lambda$.
It implies that  $\Omega_{\bar{b}}(A)<0$ when $c_1^{vert}(A)<-N$.
Therefore, $\ol{\cM_{0,2}}(\bar{P}_f, A, \{J_b\})$ is an empty set for this choice of $\Omega_{\bar{b}}$.
By cobordism invariance (cf. Lemma \ref{l:correlationmap}), the result follows.

To verify the claim, we consider the map $F=(F_H,F_S,F_T):Y \to LG \times S^{\infty} \times ET$ associated to $\bar{f}$ (see \eqref{eq:Fmap}).
We also consider the family of symplectic structures $\{\widetilde{\Omega}_{F_H(y)}^{aver}\}_{y \in Y}$ on $P_F$ obtained by pulling-back $\{\widetilde{\Omega}_{\gamma}^{aver}\}_{\gamma \in LG}$.
By construction, the family $\{\widetilde{\Omega}_{F_H(y)}^{aver}\}_{y \in Y}$ descends to $\{\Omega_{\bar{b}}'\}_{\bar{b} \in \bar{B}}$.
Over $Y \times M \times D^2_0$, $\{\widetilde{\Omega}_{F_H(y)}^{aver}\}_{y \in Y}$ is given by
\begin{align*}
\widetilde{\Omega}_{F_H(y)}^{aver}= \omega + \epsilon 2r dr \wedge d \theta
\end{align*}
Over $Y \times M \times D^2_{\infty}$, it is given by
\begin{align}
\widetilde{\Omega}_{F_H(y)}^{aver}= \omega + \beta_1^{aver}(F_H(y))+\beta_2^{aver}(F_H(y)) \label{eq:SMoothing2}
\end{align}
where $\beta_1, \beta_2$ are given by \eqref{eq:beta1}, \eqref{eq:beta2} and $\beta_i^{aver}$ is the average of $\beta_i$ over the $\hat{T} \times T$-action (see Example \ref{e:admissSym}).
The reason that $\{\widetilde{\Omega}_{F_H(y)}^{aver}\}_{y \in Y}$ is not necessarily smooth is because there is a term $\max_x H^{F_H(y)}_{\theta}(x)$ in $\beta_2(F_H(y))$ and this term might not vary smoothly with $y \in Y$.
We can pick a smooth function $c:Y \times D^2_{\infty} \to \mathbb{R}$ such that 
\begin{align}
c(y,re^{2\pi i \theta}) \ge \max_x H^{F_H(y)}_{\theta}(x) \label{eq:upsmooth}
\end{align}
 for all $y$ and $\theta$.
If we replace the term $\max_x H^{F_H(y)}_{\theta}(x)$ in $\beta_2(F_H(y))$ by $c$ and denote the resulting $2$-form by $\beta_{2,c}$, then the average $\beta_{2,c}^{aver}(F_H(y))$ over the $\hat{T} \times T$-action is smooth over $Y \times M \times D^2_{\infty}$.
We can therefore replace the term $\beta_2^{aver}(F_H(y))$ in \eqref{eq:SMoothing2} by $\beta_{2,c}^{aver}(F_H(y))$ to get a smooth family of fiberwise $2$-forms on $P_F$.
The inequality \eqref{eq:upsmooth} guarantees that this smooth family of $S^1 \times T \times T$-invariant forms descends to an admissible family of symplectic forms for $\bar{P}_f \to \bar{B}$.
Moreover, by choosing $c$ to be sufficiently $C^0$-close to $\max_x H^{F_H(y)}_{\theta}(x)$, we can achieve our claim.


\end{proof}

Let $\bar{B}$ be closed and $\bar{f}:\bar{B} \to (LG/T)_{borel,n} \subset (LG/T)_{borel}$ be smooth.
Following the overview in Section \ref{ss:overviewSeidelMap}, we consider the composition of the following maps (which plays the role analogous to $\cC_M \circ \cQ$)
and denote it by $\cS_{\bar{f},n}$.
\begin{align}
\hat{H}_*(\bar{B}) \times \hat{H}_*^{\hat{T}}(M)[q^{\pm 1}] \simeq&  \hat{H}_{*+\dim(S^1 \times T \times T)}^{S^1 \times T \times T}(Y) \times \hat{H}_*^{\hat{T}}(M)[q^{\pm 1}] \label{eq:Se1}\\
\simeq& \hat{H}^{S^1 \times T \times T \times S^1 \times T}_{*+\dim(S^1 \times T \times T)}(Y \times M)[q^{\pm 1}] \label{eq:Se2}\\
\to & \hat{H}^{S^1_{\Delta} \times T \times T_{\Delta}}_{*+\dim(S^1 \times T \times T)}(Y \times M)[q^{\pm 1}] \label{eq:Se3}\\
\simeq &\hat{H}_*((\bar{P}_{f_k})_0)[q^{\pm 1}] \label{eq:Se4}\\
\to &\hat{H}_*((\bar{P}_{f_k})_{\infty})[q^{\pm 1}] \label{eq:Se5}\\
\to & \hat{H}_{*}((S_{T,\infty})_{borel, n})[q^{\pm 1}]=\hat{H}_{*}((LG/T \times M)_{borel,n})[q^{\pm 1}] \label{eq:Se6}\\
\to  &\hat{H}_*(M_{borel,n})[q^{\pm 1}] \label{eq:Se7}
\end{align}
The isomorphisms \eqref{eq:Se1} and \eqref{eq:Se2} follow from Lemma \ref{lem: freeactions}
and the K\"unneth formula.
The map \eqref{eq:Se3} is induced by the restriction to subgroup as in \eqref{eq:mult1}.
The isomorphism \eqref{eq:Se4} follows from Lemma \ref{lem: freeactions}.
The map \eqref{eq:Se5} is $\cC_{\bar{f}}$.
The maps \eqref{eq:Se6} and \eqref{eq:Se7} are induced by the natural maps $(\bar{P}_{f_k})_{\infty} \to (S_{T,\infty})_{borel, n}$ (cf. \eqref{eq:PtoSinf})
and $(LG/T \times M)_{borel,n} \to M_{borel,n}$, respectively.



\subsection{Digression: geometric homology}

The affine Schubert stratification gives a $\hat{T}$ CW complex structure on $L_{poly}G/T$. We denote by $(L_{poly}G/T)_{\le d}$ the union of cells with dimension less than $d$.
We have (see \eqref{eq: inverselimit}, \eqref{eq:directlim})
\begin{align*}
\hat{H}^{\hat{T}}_*(L_{poly}G/T) :=& \varinjlim_d \hat{H}^{\hat{T}}_*((L_{poly}G/T)_{\leq d})\\
=&\varinjlim_d \varprojlim_n H_ {\operatorname{dim} B\hat{T}_n+ \ast}((L_{poly}G/T)_{\le d, borel,n}) 
\end{align*}
We are going to describe a model of $H_ {\operatorname{dim} B\hat{T}_n+ \ast}((L_{poly}G/T)_{\le d, borel,n})$ called geometric homology \cite{Jak98}.

Let $N$ be a finite CW complex. 
We consider triples $(B,\alpha,f)$ such that $B$ is a closed oriented manifold, $\alpha \in H^*(B,\mathbb{Z}),$
and $f:B \to N$ is a continuous map.
Two such triples $(B,\alpha,f)$ and $(B',\alpha',f')$ are called {\it equivalent} if there is an orientation preserving diffeomorphism
$\phi:B' \to B$ such that $\phi^* \alpha=\alpha'$ and $f'=f \circ \phi$.
The geometric homology $H^{geo}_*(N,\mathbb{Z})$ is defined to be the free abelian group generated by equivalent classes of triples
modulo the following relations
\begin{itemize}
\item if $B=B_1 \sqcup B_2$, then $(B,\alpha,f)=(B_1, \alpha|_{B_1},f|_{B_1})+(B_2, \alpha|_{B_2},f|_{B_2})$
\item $(B,\alpha_1+\alpha_2,f)=(B,\alpha_1,f)+(B,\alpha_2,f)$
\item if there exist an oriented manifold $B'$ with boundary $\partial B'=B$ as oriented manifolds, then for any $\alpha' \in H^*(B',\mathbb{Z})$ and $f':B'\to N$, we have $(B, \alpha'|_{B}, f'|_{B})=0$
\item let $E$ be an oriented rank $r$ vector bundle over $B$, $S(E \oplus \ul{1})$ be the sphere bundle of the direct sum of $E$ and the rank $1$ trivial bundle $\ul{1}$, $\sigma:B \to S(E \oplus 1)$ be the section coming from $(0,1) \in E \oplus \ul{1}$ and $\pi:S(E \oplus \ul{1}) \to B$ be the projection.
Then we have $(B, \alpha,f)=(S(E \oplus 1), \sigma_!(\alpha), f\circ \pi)$, where $\sigma_!: H^*(B) \to H^{*+r}(S(E \oplus 1))$ is the Gysin homomorphism.
\end{itemize} 

\begin{thm}[Proposition 2.3.2 of \cite{Jak98}]\label{t:Jak}
The map $\psi: H^{geo}_*(N,\mathbb{Z}) \to H_*(N,\mathbb{Z})$ defined by
\begin{align*}
(B, \alpha,f) \mapsto f_*(\alpha \cap [B])
\end{align*}
is well-defined and is an isomorphism.
\end{thm}

Although not explicitly stated in \cite{Jak98}, the result remains valid if we impose an additional assumption that $B$ and $E$ are smooth (but $f$ is only continuous because $N$ does not have a smooth structure).
To see this, denote this variant by $H^{geo,sm}_*(N,\mathbb{Z})$.
Since any relation in the smooth category is a relation in the topological category, there is a well-defined map
$\psi^{sm}: H^{geo,sm}_*(N,\mathbb{Z}) \to H_*(N,\mathbb{Z})$ as above.
In the proof of Proposition 2.3.2 (\cite[p.78-80]{Jak98}), both surjectivity and injectivity of $\psi$
are proved by reducing it to the smooth category using smooth approximation theorems so the same proof goes through for $H^{geo,sm}_*(N,\mathbb{Z})$.




\subsection{Main construction (continued)}

We are going to construct a map 
\begin{align*}
\varinjlim_d \varprojlim_n H^{geo,sm}_ {\operatorname{dim} B\hat{T}_n+ \ast}((L_{poly}G/T)_{\le d, borel,n}) \to \End(\hat{H}_*^{\hat{T}}(M)[q^{\pm 1}])
\end{align*}

Let $(\bar{B},\bar{\alpha},\bar{f})$ be a representative of 
$ H^{geo,sm}_ {\operatorname{dim} B\hat{T}_n+ \ast}((L_{poly}G/T)_{\le d, borel,n})$.
By Lemma \ref{l:smoothApproximation}, there is a sequence of smooth maps $\bar{f}_k:\bar{B} \to (LG/T)_{borel,n}$ which $C^0$-converges to $\bar{f}$.
For each $k$, we have a linear map (see \eqref{eq:Se1} to \eqref{eq:Se7})
\begin{align}
\cS_{\bar{f}_k,n}( PD(\bar{\alpha}), -): \hat{H}_*^{\hat{T}}(M)[q^{\pm 1}] \to \hat{H}_*(M_{borel,n})[q^{\pm 1}] \label{eq:Seidelkn}
\end{align}

\begin{lem}\label{l:SeGeoMap}
The map $\cS_{\bar{f}_k,n}( PD(\bar{\alpha}), -)$ is independent of $k$ for sufficiently large $k$.
Moreover, the $\mathbb{Z}$-linear map 
\begin{align}
H^{geo,sm}_ {\operatorname{dim} B\hat{T}_n+ \ast}((L_{poly}G/T)_{\le d, borel,n}) &\to \Hom(\hat{H}_*^{\hat{T}}(M)[q^{\pm 1}], \hat{H}_*(M_{borel,n})[q^{\pm 1}]) \label{eq:geosmSeidel1}\\
[\bar{B},\bar{\alpha},\bar{f}] & \mapsto  \lim_k \cS_{\bar{f}_k,n}( PD(\bar{\alpha}), -) \label{eq:geosmSeidel2}
\end{align}
is well-defined. In \eqref{eq:geosmSeidel1}, $\Hom$ refers to the space of degree preserving $\bk$-linear maps.
\end{lem}

\begin{proof}
For sufficiently large $k_0,k_1$, 
$\bar{f}_{k_0}$ and $\bar{f}_{k_1}$ are homotopic so
there is a continuous map $\bar{c}:[0,1] \times \bar{B} \to (LG/T)_{borel,n}$
such that $\bar{c}(t,-)=\bar{f}_{k_0}$ near $t=0$ and $\bar{c}(t,-)=\bar{f}_{k_1}$ near $t=1$.
By Lemma \ref{l:smoothApproximation}, we can assume that $\bar{c}$ is smooth.
Both $[0,1] \times \bar{P}_{\bar{f}_{k_0}}$ and $\bar{P}_{\bar{c}}$ are naturally a fiber bundle over $[0,1]$.
By choosing a diffeomorphism $[0,1] \times \bar{P}_{\bar{f}_{k_0}} \simeq \bar{P}_{\bar{c}}$ covering the identity map of $[0,1]$, we can regard $\{\cC_{\bar{c}(t,-)}\}_{t \in [0,1]}$ as a family of maps defined using a family of auxiliary data on 
$\bar{P}_{\bar{f}_{k_0}}$.
By cobordism invariance (Lemma \ref{l:correlationmap}), we have $\cC_{\bar{f}_{k_0}}=\cC_{\bar{f}_{k_1}}$.
It clearly implies that \eqref{eq:Seidelkn} is independent of $k$ as long as $\bar{f}_{k}$ is homotopic to $\bar{f}$, which is true when $k$ is large.

To show that \eqref{eq:geosmSeidel2} is well-defined, we need to check that it is independent of representatives.
There are $4$ relations to check (see the paragraph before Theorem \ref{t:Jak}).
The first two are rather straightforward.
The third follows from a standard cobordism argument.
Therefore, we only spell out the last one.

The section $\sigma:\bar{B} \to S(E \oplus \ul{1})$ induces natural inclusion maps
$\sigma^{Y,0/\infty}: (\bar{P}_{f_k })_{0/\infty} \to (\bar{P}_{f_k \circ \pi})_{0/\infty}$.
To finish the proof, it suffices to establish the following commutative diagram (the notation $[q^{\pm 1}]$ is omitted in the diagram for convenience)
\begin{equation*}
\begin{tikzcd}
\hat{H}_*(\bar{B}) \otimes \hat{H}_*^{\hat{T}}(M) \arrow{r} \arrow[swap]{d}{\sigma_* \otimes 1}
& \hat{H}_*((\bar{P}_{f_k })_0) \arrow{r}{\cC_{\bar{f}_k}} \arrow[swap]{d}{\sigma^{Y,0}_*}
& \hat{H}_*((\bar{P}_{f_k})_{\infty}) \arrow{r} \arrow[swap]{d}{\sigma^{Y,\infty}_*}
& \hat{H}_*((S_{T,\infty})_{borel,n}) \arrow[swap]{d}{=}
\\
\hat{H}_*(S(E \oplus \ul{1})) \otimes \hat{H}_*^{\hat{T}}(M) \arrow{r}
& \hat{H}_*((\bar{P}_{f_k \circ \pi})_0) \arrow{r}{\cC_{\bar{f}_k \circ \pi}}
& \hat{H}_*((\bar{P}_{f_k \circ \pi})_{\infty}) \arrow{r}
& \hat{H}_*((S_{T,\infty})_{borel,n})
\end{tikzcd}
\end{equation*}
where the horizontal maps from the first column to the second column come from \eqref{eq:Se1} to \eqref{eq:Se4}.
The commutativity for the square on the left comes from functorality of the operations in \eqref{eq:Se1} to \eqref{eq:Se4}.
The commutativity for the square on the right comes from the fact that 
the bottom horizontal map factors through $\hat{H}_*((\bar{P}_{f_k})_{\infty})$
and $\sigma^{Y,\infty}$ is an embedding.
Finally, to verify the  commutativity for the square on the middle, we choose the auxiliary data on 
$\bar{P}_{f_k \circ \pi}$ to be the pull-back of an auxiliary data on $\bar{P}_{f_k}$.
In this case, the  auxiliary data on $\bar{P}_{f_k \circ \pi}$  is regular if and only if that on $\bar{P}_{f_k}$ is regular.
As a result, $\cM_{0,2}(\bar{P}_{f_k \circ \pi}, A, \{J_{\bar{b}}\})$ is a sphere-bundle over
$\cM_{0,2}(\bar{P}_{f_k}, A, \{J_{\bar{b}}\})$
and the commutativity of the middle square follows easily.

\end{proof}

\begin{thm}\label{t:invSeidel}
The following diagram commutes
\begin{equation}\label{eq:ComGeo}
\begin{tikzcd}
H^{geo,sm}_ {\operatorname{dim} B\hat{T}_{n+1}+ \ast}((L_{poly}G/T)_{\le d, borel,n+1}) \arrow{d} \arrow{r}
& \Hom(\hat{H}_*^{\hat{T}}(M)[q^{\pm 1}], \hat{H}_*(M_{borel,n+1})[q^{\pm 1}]) \arrow{d}
\\
H^{geo,sm}_ {\operatorname{dim} B\hat{T}_n+ \ast}((L_{poly}G/T)_{\le d, borel,n})  \arrow{r}
& \Hom(\hat{H}_*^{\hat{T}}(M)[q^{\pm 1}], \hat{H}_*(M_{borel,n})[q^{\pm 1}])
\end{tikzcd}
\end{equation}
where the horizontal maps are defined by \eqref{eq:geosmSeidel1}.

Moreover, taking the inverse limit gives a well-defined map
\begin{align}
 \varprojlim_n H^{geo,sm}_ {\operatorname{dim} B\hat{T}_n+ \ast}((L_{poly}G/T)_{\le d, borel,n}) \to \End(\hat{H}_*^{\hat{T}}(M)[q^{\pm 1}]) \label{eq:invSeidel}
\end{align}
\end{thm}

We shall make use of the following lemma.
\begin{lem}\label{l:smoothing2}
Let $B_0,B_1$ be finite dimensional smooth compact manifolds and $\Pi:E \to B_1$ be a topological fiber bundle ($\Pi$ is not assumed to be smooth and $E$ is not assumed to be finite dimensional).
Let $f:B_0 \to E$ be a continuous map.
Then there exist a sequence of continuous maps $(f_k:B_0 \to E)_k$ such that $f_k$ converges to $f$ in $C^0$ and $\Pi \circ f_k$ is smooth for all $k$.

Moreover, if $B_2$ is a smooth closed submanifold of $B_1$, we can further assume that  $\Pi \circ f_k$ is transversal to $B_2$ for all $k$.
\end{lem}

\begin{proof}
By the smooth approximation theorem (e.g. by taking covolution), there exist a continuous map
$c:[0,1] \times B_0 \to B_1$
such that $c(0,-)=\Pi \circ f$ and $c(t,-)$ is smooth for all $t \neq 0$.
The continuity of $c$ and compactness of $B_0$ guarantees that $c(t,-)$ converges to $c(0,-)$ in $C^0$ as $t$ goes to $0$.
Therefore, it suffices to show that for a sequence $(t_k)_k$ converging to $0$, we can find $(f_k:B_0 \to E)_k$
such that $c(t_k,-)=\Pi \circ f_k$ for all $k$ and $f_k$ converges to $f$ in $C^0$.

Let $\Pi_c: c^*E \to [0,1] \times B_0$ be the pull-back fiber bundle.
Let  $(c^*E)_0:= \Pi_c^{-1}(\{0\} \times B_0)$ and $\Pi_{c,0}:=\Pi_c|_{(c^*E)_0}$.
The map $f:B_0 \to E$ can be canonically identified as a section of $\Pi_{c,0}$.
Since $[0,1] \times B_0$ is homotopic to $B_0$, we can extend this section to a section, denoted by $\sigma$, of $\Pi_{c}$.
For any $t \in [0,1]$, $\sigma|_{\{t\} \times B_0}$ can be canonically identitified with a continuous map
$f_t:B_0 \to E$ such that $\Pi \circ f_t=c(t,-)$.
By continuity of $\sigma$ and compactness of $B_0$, $f_t$ converges to $f$ in $C^0$ as $t$ goes to $0$.
This proves the first part the lemma.

For the second part, we can first pick
$c:[0,1] \times B_0 \to B_1$ such that $c(0,-)=\Pi \circ f$
and there is a sequence $t_k$ going to $0$ with the property that $c(t_k,-)$ is transversal to $B_2$ for all $k$.
Then we run the rest of the arguments.
\end{proof}

\begin{proof}[Proof of Theorem \ref{t:invSeidel}]
Let $(\bar{B}_{n+1},\bar{\alpha}_{n+1},\bar{f}_{n+1})$ be a representative of a class in 
$H^{geo,sm}_ {\operatorname{dim} B\hat{T}_{n+1}+ \ast}((L_{poly}G/T)_{\le d, borel,n+1})$.
Denote the projection $(L_{poly}G/T)_{\le d, borel,n+1}\to B\hat{T}_{n+1}$ by $\Pi$.
By possibly perturbing $\bar{f}_{n+1}$, we can assume that $\Pi \circ \bar{f}_{n+1}$ is smooth and transversal to $B\hat{T}_{n}$ (see Lemma \ref{l:smoothing2})
so that the preimage of $B\hat{T}_{n}$ is a smooth submanifold, denoted by $\bar{B}_n$, of $\bar{B}_{n+1}$.
Then $(\bar{B}_{n},\bar{\alpha}_{n}:=\bar{\alpha}_{n+1}|_{\bar{B}_{n}},\bar{f}_{n}:=\bar{f}_{n+1}|_{\bar{B}_{n}})$ is a representative of the restriction of the class $[\bar{B}_{n+1},\bar{\alpha}_{n+1},\bar{f}_{n+1}]$
to $H^{geo,sm}_ {\operatorname{dim} B\hat{T}_{n}+ \ast}((L_{poly}G/T)_{\le d, borel,n})$.

By Lemma \ref{l:smoothApproximation}, we have a sequence of smooth maps $\bar{f}_{n+1,k}:\bar{B}_{n+1} \to (LG/T)_{borel,n+1}$ converges in $C^0$ to  $\bar{f}_{n+1}$.
The smoothness of $\Pi \circ \bar{f}_{n+1}$ enable us to find $\bar{f}_{n+1,k}$ such that $\Pi \circ \bar{f}_{n+1,k}=\Pi \circ \bar{f}_{n+1}$ for all $k$.
In particular, $\Pi \circ \bar{f}_{n+1,k}$ is transversal to $B\hat{T}_{n}$
and $\bar{B}_{n,k}:=(\Pi \circ \bar{f}_{n+1,k})^{-1}(B\hat{T}_{n})$ is the same as $\bar{B}_{n}$.
Let $\bar{\alpha}_{n,k}:=\bar{\alpha}_{n+1}|_{\bar{B}_{n,k}}$ and $\bar{f}_{n,k}:=\bar{f}_{n+1}|_{\bar{B}_{n,k}}$.
The image of $[\bar{B}_{n},\bar{\alpha}_{n},\bar{f}_{n}]$ under the map \eqref{eq:geosmSeidel1}
is $\lim_k \cS_{\bar{f}_{n,k},n}(PD(\bar{\alpha}_{n,k}),-)$.
To show that $\lim_k \cS_{\bar{f}_{n+1,k},n+1}(PD(\bar{\alpha}_{n+1}),-)$ maps to $\lim_k \cS_{\bar{f}_{n,k},n}(PD(\bar{\alpha}_{n,k}),-)$ under the vertical map on the right of \eqref{eq:ComGeo}, it suffices to establish the following commutative diagram  (the notation $[q^{\pm 1}]$ is omitted in the diagram for convenience).
\begin{equation*}
\begin{tikzcd}
\hat{H}_*(\bar{B}_{n+1}) \otimes \hat{H}_*^{\hat{T}}(M) \arrow{r} \arrow{d}
& \hat{H}_*((\bar{P}_{f_{n+1,k} })_0) \arrow{r}{\cC_{\bar{f}_{n+1,k}}} \arrow{d}
& \hat{H}_*((\bar{P}_{f_{n+1,k}})_{\infty}) \arrow{r} \arrow{d}
& \hat{H}_*((S_{T,\infty})_{borel,n+1}) \arrow{d}
\\
\hat{H}_*(\bar{B}_{n,k}) \otimes \hat{H}_*^{\hat{T}}(M) \arrow{r}
& \hat{H}_*((\bar{P}_{f_{n,k}})_0) \arrow{r}{\cC_{\bar{f}_{n,k}}}
& \hat{H}_*((\bar{P}_{f_{n,k}} )_{\infty}) \arrow{r}
& \hat{H}_*((S_{T,\infty})_{borel,n})
\end{tikzcd}
\end{equation*}
where all the vertical maps are given by Poincare dual of the pull-back on cohomology.
The commutivity of the square on the left on the right follows from the functorality.
To establish the commutivity of the square in the middle, we choose the auxiliary data on 
$\bar{P}_{f_{n,k}}$ to be the restriction of the auxiliary data on $\bar{P}_{f_{n+1,k}}$.
We can ensure that both are regular simultaneously.
In this case, the moduli space 
$\cM_{0,2}(\bar{P}_{f_{n,k}}, A, \{J_{\bar{b}}\})$ is the fiber product between
$\cM_{0,2}(\bar{P}_{f_{n+1,k}}, A, \{J_{\bar{b}}\}) \to B\hat{T}_{n+1}$
and $B\hat{T}_{n} \to B\hat{T}_{n+1}$
so the commutivity of the middle square follows.

Finally, to conclude that we have a well-defined map \eqref{eq:invSeidel}, we need to show that 
for any $\{(\bar{B}_{n},\bar{\alpha}_{n},\bar{f}_{n})\}_{n \in \mathbb{N}}$ representing a class in the LHS of \eqref{eq:invSeidel} and for any pure degree class $x \in \hat{H}_*^{\hat{T}}(M)$, the limit
$\lim_n \lim_k \cS_{\bar{f}_{n,k},n}(PD(\bar{\alpha}_{n}),x)$ exists in $\hat{H}_*^{\hat{T}}(M)[q^{\pm 1}]$.
In other words, we need to bound the power of $q$ in $\lim_k \cS_{\bar{f}_{n,k},n}(PD(\bar{\alpha}_{n}),x)$ as $n$ goes to $\infty$.
Thanks to Lemma \ref{r:inverselimit}, the limit exists.
\end{proof}

It is clear that the map \eqref{eq:invSeidel} is compatible with taking direct limit in $d$.
Therefore, we obtain our equivariant Seidel map
\begin{align*}
\cS: \hat{H}^{\hat{T}}_*(L_{poly}G/T) \simeq \varinjlim_d \varprojlim_n H^{geo,sm}_ {\operatorname{dim} B\hat{T}_n+ \ast}((L_{poly}G/T)_{\le d, borel,n}) \to \End(\hat{H}_*^{\hat{T}}(M)[q^{\pm 1}])
\end{align*}

Let $E^0\hat{T}$ and $E^1\hat{T}$ be two different models for the classifying space of $\hat{T}$.
Then $E^2\hat{T}:=E^0\hat{T} \times E^1\hat{T}$ is also a model for the classifying space of $\hat{T}$.
Let $E^2\hat{T}_n:=E^0\hat{T}_n \times E^1\hat{T}_n$.
For any finite $\hat{T}$ CW complex $N$, let $N_{E^i,n}:=(N \times E^i\hat{T}_n)/\hat{T}$.
The natural projections
$N_{E^2,n}  \to N_{E^i,n}$
for $i=0,1$ and $n \in \mathbb{N}$ induce isomorphisms
\begin{align*}
\phi_{E^iE^2}:\varprojlim_n H^*(N_{E^i,n}) &\simeq \varprojlim_n H^*(N_{E^2,n}) 
\end{align*}
Let $\tilde{\phi}_{E^0E^1}:=\phi_{E^1E^2}^{-1} \circ \phi_{E^0E^2}$.
These isomorphisms are canonical in the sense that $\tilde{\phi}_{EE}=id$ and $\tilde{\phi}_{E^1E^2} \circ \tilde{\phi}_{E^0E^1}=\tilde{\phi}_{E^0E^2}$ for any models $E^0,E^1,E^2$.
The next result verifies that $\cS$ is independent of the model of the classifying space.

\begin{prop}[independence of classifying model]\label{p:indepModel}
The equivariant Seidel map $\cS$ intertwines the canonical isomorphisms between different models of the classifying space $B\hat{T}_\rho$ from \eqref{eq:classifying}.
\end{prop}

\begin{proof}[Sketch of proof]
Let $\cS^i$ be the equivariant Seidel map defined using the model $E^i\hat{T}$ for $i=0,1,2$ as above.
We want to show that $\cS^0=\cS^2(=\cS^1)$ under the identifications by the canonical isomorphisms.

It suffices to check that the following diagram commutes for any $\bar{f}^0$ and $n$
\begin{equation}\label{eq:fiberCom}
\begin{tikzcd}
 H^{\dim(\bar{B}^0)-*}(\bar{B}^0) \otimes \varprojlim_n H^{\dim(M)-*}(M_{E^0,n})[q^{\pm 1}] \arrow{r}{\cS_{\bar{f}^0,n}} \arrow{d}
&  H^{\dim(\bar{B}^0)+\dim(M)-*}(M_{E^0,n})[q^{\pm 1}] \arrow{d}
\\
 H^{\dim(\bar{B}^0)-*}(\bar{B}^2) \otimes \varprojlim_n H^{\dim(M)-*}(M_{E^2,n})[q^{\pm 1}] \arrow{r}{\cS_{\bar{f}^2,n}}
&  H^{\dim(\bar{B}^0)+\dim(M)-*}(M_{E^2,n})[q^{\pm 1}]
\end{tikzcd}
\end{equation}
In the diagram, $\bar{f}^0:\bar{B}^0 \to (LG/T)_{E^0,n}$ is smooth, $\bar{B}^2 \to \bar{B}^0$ is the pull-back of
$(LG/T)_{E^2,n} \to (LG/T)_{E^0,n}$ along  $\bar{f}^0$ and $\bar{f}^2:\bar{B}^2 \to (LG/T)_{E^2,n}$ is the associated map covering $\bar{f}^0$.
It is immediate to check that $\bar{f}^2$ is also smooth so $\cS_{\bar{f}^0,n}$ and $\cS_{\bar{f}^2,n}$ are well-defined.
The vertical arrows in the diagram are induced by pull-back of the natural maps.
The horizontal arrow on top is the cohomology version of $\cS_{\bar{f}^0,n}$ (cf. \eqref{eq:Se1} to \eqref{eq:Se7})
and the horizontal arrow on the bottom is the cohomology version of $\cS_{\bar{f}^2,n}$ shifted by degree 
$\dim(\bar{B}^0)-\dim(\bar{B}^2)$.
Similar to the proof of Lemma \ref{l:SeGeoMap} and Theorem \ref{t:invSeidel}, the commutativity of the diagram \eqref{eq:fiberCom} boils down to the commutativity of
\begin{equation*}
\begin{tikzcd}
& \hat{H}^*((\bar{P}_{f^0 })_0)[q^{\pm 1}] \arrow{r}{\cC_{\bar{f}^0}} \arrow{d}
& \hat{H}^*((\bar{P}_{f^0})_{\infty})[q^{\pm 1}]  \arrow{d}
\\
& \hat{H}^*((\bar{P}_{f^2})_0)[q^{\pm 1}] \arrow{r}{\cC_{\bar{f}^2}}
& \hat{H}^*((\bar{P}_{f^2} )_{\infty})[q^{\pm 1}] 
\end{tikzcd}
\end{equation*}
which in turn folllows from the fact that $\bar{f}^2$ is a bundle map covering $\bar{f}^0$ and we can choose auxiliary data on $\bar{P}_{f^2}$ that are pull-back from the auxiliary data on $\bar{P}_{f^0}$.


The result now follows from Lemma \ref{l:SeGeoMap}, Theorem \ref{t:invSeidel}
and passing to the inverse limit in $n$.

\end{proof}

We end this section with the following basic properties.

\begin{lem}[Identity] \label{ex:identity} Let $\bar{B}_n=B\hat{T}_n$ and $\bar{f}: \bar{B}_n \to (LG/T)_{borel,n}$ be the inclusion $\lbrace [id] \rbrace \times B\hat{T}_n \subset (LG/T)_{borel,n}$.  Then $\cC_{\bar{f}}: \hat{H}_*(M_{borel,n})[q^{\pm 1}] \to \hat{H}_*(M_{borel,n})[q^{\pm 1}]$ is the identity.   
As a result, $\cS([id], -):\hat{H}_*^{\hat{T}}(M)[q^{\pm 1}] \to \hat{H}_*^{\hat{T}}(M)[q^{\pm 1}]$ is the identity map.
\end{lem} 
\begin{proof} 
We have that $\bar{P}_f \simeq M_{borel, n} \times \mathbb{P}^1$.  Working in a fiber of $M_{borel, n} \times \mathbb{P}^1$ over some point $b \in \bar{B}_n$,  the  count of two-pointed spheres is then equivalent to by quantum multiplication by the fundamental class $[M]$ (which is identity because $[M]$ is the unit).  We can then choose auxiliary data to be independent of $\bar{B}_n$ so
$\cC_{\bar{f}}=\operatorname{id}.$   

The identity in $\hat{H}^{\hat{T}}_*(L_{poly}G/T) $ is represented by $\lbrace [id] \rbrace \times B\hat{T} \subset (LG/T)_{borel}$ so it follows easily that $\cS([id], -)$ is the identity map.
\end{proof}

We have an algebra embedding $\hat{H}^{\hat{T}}_*(pt) \to \hat{H}^{\hat{T}}_*(L_{poly}G/T)$ induced by mapping the point to $[id] \in L_{poly}G/T$.
It induces both a left and a right $\hat{H}^{\hat{T}}_*(pt) $-action on $\hat{H}^{\hat{T}}_*(L_{poly}G/T)$.
On the other hand, we have a left $\hat{H}^{\hat{T}}_*(pt)$-action on $\hat{H}^{\hat{T}}_*(M) \simeq H^{\dim(M)-*}_{\hat{T}}(M)$ via the identification $\hat{H}^{\hat{T}}_*(pt)=H^{-*}(B\hat{T})$.

\begin{lem}[$\hat{H}^{\hat{T}}_*(pt)$-linearity]\label{l:linearity}
For any $\beta \in \hat{H}^{\hat{T}}_*(pt)$ and $\alpha \in \hat{H}^{\hat{T}}_*(L_{poly}G/T)$, we have $\cS(\beta \cdot \alpha,-)=\beta \cdot \cS(\alpha,-)$.
\end{lem}

\begin{proof}


Suppose that $\alpha$ is a represented by a sequence $\{(\bar{B}_{n},\bar{\alpha}_{n},\bar{f}_{n})\}_{n \in \mathbb{N}}$ such that $\bar{f}_{n}:\bar{B}_{n} \to (L_{poly} G/T)_{\le d, borel,n}$ is smooth for all $n$.
Let $\Pi_n: (L_{poly} G/T)_{\le d, borel,n} \to B\hat{T}_n$ be the projection and $\beta \in H^*(B\hat{T}_n)$.
We claim that the 
sequence $\{(\bar{B}_{n},\bar{\alpha}_{n} \cup (\Pi_n \circ \bar{f}_{n})^* \beta,\bar{f}_{n})\}_{n \in \mathbb{N}}$
represent the class $\beta \cdot \alpha \in \hat{H}^{\hat{T}}_*(L_{poly}G/T)$.

To see this, recall that the algebra structure on $\hat{H}_*^{\hat{T}}(LG/T) $ is defined to be the composition of the 
K\"unneth map \eqref{eq:kunnethR} and the map induced by the pointwise multiplication \eqref{eq:pointmult}.
Both of these maps are left $R$-linear.
Therefore, $\beta \cdot \alpha$ is the inverse limit of $(\bar{f}_n)_*( \bar{\alpha}_n \cap [\bar{B}_n]) \cap \Pi_n^* \beta$ over $n$. 
By the functorality of cap product, it is the same as 
 $(\bar{f}_n)_*( (\bar{\alpha}_n \cup (\Pi_n \circ \bar{f}_{n})^* \beta) \cap [\bar{B}_n])$.

Inspecting \eqref{eq:Se1}-\eqref{eq:Se7}, we see that
$\cS_{\bar{f}_{n},n}(\bar{\alpha}_{n} \cup (\Pi_n \circ \bar{f}_{n})^* \beta,-)=\cS_{\bar{f}_{n},n}(\bar{\alpha}_{n},-) \cap \Pi_{M,n}^*\beta$,
where $\Pi_{M,n}: M_{borel,n} \to  B\hat{T}_n$ is the projection.
By passing to the inverse limit over $n$, it shows that $\cS(\beta \cdot \alpha,-)=\beta \cdot \cS(\alpha,-)$.

\end{proof}


\section{Properties of Shift operators}
\subsection{Main properties}  \vskip 5 pt

\subsubsection{Module property}\label{s:module}
It will be convenient in this section to denote by $\rho_{\operatorname{id}}$ the obvious $\hat{T}$ action on $M$ given by $\rho_{\operatorname{id}} (a, t, m) = t \cdot m$.  Any fixed point of the form $\sigma [w] \in LG/T$ defines an equivarant cycle $$B_{\sigma [w]} \to LG/T_{borel} $$   In view of \eqref{eq: Taction0},  the corresponding divisor at zero $S_{T,0}$ is $\hat{T}$-equivariantly identified with $M$ equipped with the twisted action:  \begin{align*} \rho_{\sigma [w]} (a, t,m) = \rho_{\operatorname{id}} (a,  w^{-1} \sigma(a)^{-1}t w,  m)= \rho_{\operatorname{id}} (\mathcal{A}_{(\sigma[w])^{-1}}(a,t),m) \end{align*}  We denote the cohomology with respect to this twisted action by $H^*_{\hat{T},\rho_{\sigma [w]}}(M)$ and for simplicity we keep $H^*_{\hat{T}, \rho_{\operatorname{id}}}(M):=H^*_{\hat{T}}(M)$.  The map \begin{align*} \mathcal{C}_{\sigma[w]}: H^*_{\hat{T},\rho_{\sigma [w]}}(M)[q^{\pm 1}] \to H^*_{\hat{T}}(M)[q^{\pm 1}]  \end{align*} 
defined by \eqref{eq:Se5} is a map of $H^*(B\hat{T})$-modules (compare \cite[\S 3.16]{LJ21}).  We will let $H^*_{\hat{T}}(M)_{(\sigma [w])}$ be a copy of $H^*_{\hat{T}}(M)$  equipped with the twisted $R$-module structure \begin{align} \label{eq:twistedR} g \cdot_{\sigma w} \beta = \mathcal{A}_{(\sigma [w])^{-1}}(g) \cdot \beta.  \end{align} where $g \in R$, $\beta \in H^*_{\hat{T}}(M)$ and $\mathcal{A}_{(\sigma [w])^{-1}}(g)$ is defined in  \eqref{eq:actfun}.  In the present context,  the Kunneth map \eqref{eq:cQ}(or more precisely the component corresponding to $\sigma[w]$) \begin{align} \label{eq:Kunneth2} \cQ: H^*_{\hat{T}}(M) \cong H^*_{\hat{T},\rho_{\sigma [w]}}(M) \end{align} is simply the pull-back map in equivariant cohomology along the automorphism $\mathcal{A}_{(\sigma[w])^{-1}}: \hat{T} \to \hat{T}.$ This means that the shift-operator corresponding to this component is an $R$-linear map \begin{align*} \mathcal{S}_{\sigma [w]}: H^*_{\hat{T}}(M)_{(\sigma [w])}[q^{\pm 1}] \to H^*_{\hat{T}}(M)[q^{\pm 1}]. \end{align*}   Passing to the fraction field,  we obtain 
\begin{align} \label{eq:FRotimes} F_R \otimes \mathcal{S}_{\sigma [w]}: F_R \otimes_{R} H^*_{\hat{T}}(M)_{(\sigma [w])}[q^{\pm 1}] \to F_R\otimes_{R} H^*_{\hat{T}}(M)[q^{\pm 1}] \end{align}

Let $M^{T}$ denote the fixed point locus of the $T$-action on $M$.  By the Atiyah-Bott localization theorem,  restriction to the fixed point locus gives canonical identifications of $F_R \otimes_{R} H^*_{\hat{T}}(M)$ and $F_R \otimes_{R} H^*_{\hat{T}}(M)_{(\sigma [w])}$ with $F_R \otimes_{\mathbf{k}} \hat{H}_*(M^{T})$ as $\mathbf{k}$-modules.  Using these identifications,  we obtain a map: \begin{align} \label{eq:sigmaw} \mathcal{S}_{\sigma [w]}^{F}: F_R \otimes_{\mathbf{k}} \hat{H}_*(M^{T})[q^{\pm 1}]  \to F_R \otimes_{\mathbf{k}} \hat{H}_*(M^{T})[q^{\pm 1}] \end{align}

This map is $F_R$-linear if the source is given the $\sigma [w]$-twisted $F_R$-module structure from  \eqref{eq:twistedR}.

 We will make use of the following algebraic lemma: 
\begin{lem}\label{l:basic1} Suppose $i_A: A_0 \hookrightarrow A_1$ is an embedding of $\mathbf{k}$-algebras and $i_M: M_0 \hookrightarrow M_1$ is an embedding of $\mathbf{k}$-vector spaces.  Let $S_0: A_0\otimes_\mathbf{k} M_0 \to M_0$ and   $S_1: A_1\otimes_\mathbf{k} M_1 \to M_1$ be $\mathbf{k}$-linear maps such that the following diagram commutes:\[
\begin{tikzcd}
A_0 \otimes_\mathbf{k} M_0 \arrow{r}{S_0} \arrow[swap]{d}{i_A\otimes i_M}& M_0 \arrow{d}{i_M} \\
A_1\otimes_\mathbf{k} M_1 \arrow{r}{S_1} & M_1
\end{tikzcd}
\]
 If $S_1$ defines a left-module structure on $M_1$ over $A_1$,  then $S_0$ defines a left-module structure on $M_0$ over $A_0.$  \end{lem} 
\begin{proof} 
This follows from simple diagram chasing.   
\begin{align*}
i_M(S_0(a_1a_2,m))=&S_1(i_A(a_1a_2),i_M(m))=S_1(i_A(a_1)i_A(a_2),i_M(m))\\
=&S_1(i_A(a_1),S_1(i_A(a_2),i_M(m)))=S_1(i_A(a_1),i_M(S_0(a_2,m)))\\
=&i_M(S_0(a_1,S_0(a_2,m)))
\end{align*}
The injectivity of $i_M$ implies that $S_0(a_1a_2,m)=S_0(a_1,S_0(a_2,m))$.
\end{proof} 

We have the following generalization of  Lemma \ref{ex:identity}.  

\begin{lem} \label{ex:Weyl} Let $w \in N_G(T)$ be a representative of $[w] \in W$  and  $\bar{f}': \bar{B} \to (LG/T)_{borel,n}$ is smooth and $\bar{f}=w \bar{f}'$,  then $\cS_{\bar{f},n}$ is given by $([w])_* \circ \cS_{\bar{f}',n}$.  Similarly, if $\bar{f}': \bar{B} \to (LG/T)_{borel,n}$ is smooth
and $\bar{f}=\bar{f}'w$, then $\cS_{\bar{f},n}$ is given by $\cS_{\bar{f}',n} \circ ([w])_*$. \end{lem} 
\begin{proof} 
Choose a model for the universal bundle $E\hat{T}$ which is also a universal bundle for $S^1\times G$ and form the classifying space $B\hat{T}$ using this model.  It follows that the Weyl group $W$ acts directly on $B\hat{T}$ and $M_{borel}.$  Let $\bar{B}_n=B\hat{T}_n$ and $\bar{f}: \bar{B}_n \to (LG/T)_{borel,n}$ be the inclusion $\lbrace [w] \rbrace \times B\hat{T}_n \subset (LG/T)_{borel,n}$.
The equivalence relation in \eqref{eq:Seidelbun} reduces to $(y,x,e^{2\pi i \theta})_0 \sim (y,wx,e^{2\pi i \theta})_\infty$.
Therefore, the map
\begin{align}
(y,x,e^{2\pi i \theta})_0 &\mapsto  (y,x,e^{2\pi i \theta})_0 \label{eq:Weyl1}\\
(y,x,e^{2\pi i \theta})_\infty&\mapsto  (y,w^{-1}x,e^{2\pi i \theta})_{\infty} \label{eq:Weyl2}
\end{align}
is an isomorphism from $P_F$ to $Y \times M  \times \mathbb{P}^1$, so $i_w: \bar{P}_f \simeq M_{borel, n} \times \mathbb{P}^1$.  Note that this isomorphism restricted to the divisor $S_{T,0}$ intertwines the Kunneth maps from  \eqref{eq:Kunneth2} i.e.  there is a commutative triangle:

\begin{tikzcd}[column sep=small]
  H^*_{\hat{T}}(M) \arrow{r}{\cQ_{[w]}}  \arrow{rd}{\cQ_{[\operatorname{id}]}=\operatorname{id}} 
  & H^*_{\hat{T},\rho_{[w]}}(M) \arrow{d}{(i_w)_*} \\
    & H^*_{\hat{T}}(M)
\end{tikzcd}
 The argument of Lemma \ref{ex:identity} then implies that $\cS_{\bar{f},n}=([w])_*$. 

 More generally,  let  $\bar{f}': \bar{B} \to (LG/T)_{borel,n}$ be smooth and $\bar{f}=w\bar{f}'$.
Let $F=(F_H,F_{\hat{T}}):Y_{F} \to LG \times E\hat{T}$ be the pul-back along $\bar{f}$.
It looks different from (see \eqref{eq:Fmap}) because we don't use a split model for $E\hat{T}$.
Let $\phi_{w^{-1}}: LG \times E\hat{T}$ be given by $(\gamma, z) \mapsto (w^{-1}\gamma, w^{-1}z)$, which covers the action of $[w^{-1}]$ on $(LG/T)_{borel}$.
Let $F'=\phi_{w^{-1}} \circ F=(F'_H, F'_{\hat{T}})$, which covers the map $\bar{f}'=w^{-1}\bar{f}$.
It satisfies
\begin{align}
F'((a,t_1,t_2)y)(\theta)=&(w^{-1}t_1F_H(y)(\theta-a) t_2^{-1}, w^{-1}(a,t_1)F_{\hat{T}}(y) )\\
=&(w^{-1}t_1w w^{-1}F_H(y)(\theta-a) t_2^{-1}, (a,w^{-1} t_1 w) w^{-1}F_{\hat{T}}(y) ) \\
=&((a, w^{-1} t_1 w, t_2)F'(y))(\theta)\\
=&((\mathcal{A}_{w^{-1}}(a,t_1), t_2)F'(y))(\theta)  \label{eq:tWisted}
\end{align}
In other words, $F'$ is equivariant with respect to the $w^{-1}$-twisted action on $LG \times E\hat{T}$.
Let $P^{tw}_{F'}=(Y_F \times M \times D^2_0 \cup Y_F \times M \times D^2_{\infty})/ \sim$ where the equivalent relation is given by $(y,x,e^{2\pi i \theta})_0 \sim (y, F'_H(y)(\theta)(x), e^{2\pi i \theta})_{\infty}$.
We can define a diffeomorphism $P_{F} \to P^{tw}_{F'}$ via \eqref{eq:Weyl1} \eqref{eq:Weyl2}.
It is $\hat{T} \times T$ equivariant with respect to the standard action on $P_{F}$ and the following twisted action on 
$P^{tw}_{F'}$
\begin{align*}
(a,t_1,t_2)\cdot (y,x,e^{2\pi i \theta})_0 &=  ((a,t_1,t_2) \cdot y,t_2 x,e^{2\pi i (\theta+a)})_0 \\
(a,t_1,t_2) \cdot (y,x,e^{2\pi i \theta})_\infty&=  ((a,t_1,t_2) \cdot y,w^{-1}t_1w x,e^{2\pi i (\theta+a)})_{\infty} 
\end{align*}
Note that this twisted action is well-defined on $P^{tw}_{F'}$ because of \eqref{eq:tWisted}.
Taking the quotient by the $S^1 \times T \times T$ action, we get an isomorphism
$i_w: \bar{P}_f \simeq \bar{P}^{tw}_{f'}$. As a result, we have the following commutative diagram (the notation $[q^{\pm 1}]$ is omitted for convenience)
\begin{equation*}
\begin{tikzcd}
\hat{H}_*(\bar{B}) \otimes \hat{H}_*^{\hat{T}}(M) \arrow{r} \arrow{d}{ =}
& \hat{H}^{S^1 \times T\times T}_*(Y_F \times M) \simeq \hat{H}_*((\bar{P}_{f})_0) \arrow{r}{\cC_{\bar{f}}} \arrow{d}{((i_w)_{0})_*}
& \hat{H}_*((\bar{P}_{f})_{\infty}) \arrow{d}{((i_w)_{\infty})_*}
\\
\hat{H}_*(\bar{B}) \otimes \hat{H}_*^{\hat{T}}(M) \arrow{r}
& \hat{H}^{S^1 \times T\times T}_*(Y_F \times M) \simeq \hat{H}_*((\bar{P}^{tw}_{f'})_0) \arrow{r}{\cC^{tw}_{\bar{f}'}}
& \hat{H}_*((\bar{P}^{tw}_{f'} )_{\infty}) 
\end{tikzcd}
\end{equation*}
The map $((i_w)_{0})_*$ is the identity map but $((i_w)_{\infty})_*$ is induced by $w^{-1}$ on $M$.

Now we want to compare $\cC^{tw}_{\bar{f}'}$ with $\cC_{\bar{f}'}$.
Let $Y_{F'}=Y_F$ but we equip it with a different $\hat{T} \times T$-action such that $F'$ is (untwisted) $\hat{T} \times T$-equivariant.
Even though the identity map $Y_F \times M \to Y_{F'} \times M$ is not $\hat{T} \times T$, it descends to a well-defined map to the quotient (just like $\phi_{w^{-1}}$ is not $\hat{T} \times T$-equivariant but it descends  to an action on $(LG/T)_{borel}$).
Similarly, the identity induces an isomorphism $\bar{P}^{tw}_{f'} \simeq \bar{P}_{f'}$ even though it does not come form an $\hat{T} \times T$-equivariant isomorphism from $P^{tw}_{F'}$ to $P_{F'}$.
Therefore, we have another commutative diagram (the notation $[q^{\pm 1}]$ is omitted for convenience)
\begin{equation*}
\begin{tikzcd}
\hat{H}_*(\bar{B}) \otimes \hat{H}_*^{\hat{T}}(M) \arrow{r} \arrow{d}{=}
& \hat{H}^{S^1 \times T\times T}_*(Y_F \times M) \simeq \hat{H}_*((\bar{P}^{tw}_{f'})_0) \arrow{d}{\simeq} \arrow{r}{\cC^{tw}_{\bar{f}'}}
& \hat{H}_*((\bar{P}^{tw}_{f'} )_{\infty}) \arrow{d}{\simeq}
\\
\hat{H}_*(\bar{B}) \otimes \hat{H}_*^{\hat{T}}(M) \arrow{r}
& \hat{H}^{S^1 \times T\times T}_*(Y_{F'} \times M) \simeq \hat{H}_*((\bar{P}_{f'})_0) \arrow{r}{\cC_{\bar{f}'}}
& \hat{H}_*((\bar{P}_{f'} )_{\infty}) 
\end{tikzcd}
\end{equation*}
Combining the two commutative diagram, it implies that $\cS_{\bar{f},n}$ is given by $([w])_* \circ \cS_{\bar{f}',n}$.


Similarly, if $\bar{f}': \bar{B} \to (LG/T)_{borel,n}$ is smooth
and $\bar{f}=\bar{f}'w$, then $\bar{P}_f \simeq \bar{P}_{f'}$ via the maps
\begin{align*}
(y,x,e^{2\pi i \theta})_0 &\mapsto  (y,wx,e^{2\pi i \theta})_0 \\
(y,x,e^{2\pi i \theta})_\infty&\mapsto  (y,x,e^{2\pi i \theta})_{\infty} 
\end{align*}
It implies that $\cS_{\bar{f},n}$ is given by $\cS_{\bar{f}',n} \circ ([w])_*$.
\end{proof}

We are now in a position to prove the module property of shift-operators.   

\begin{thm}\label{t:proof1} The action $$ \cS: \hat{H}_*^{\hat{T}}(LG/ T) \otimes_\mathbf{k}\hat{H}_*^{\hat{T}}(M)[q^{\pm 1}] \to \hat{H}_*^{\hat{T}}(M)[q^{\pm 1}] $$ defines a module structure. \end{thm} 
\begin{proof} 
We wish to apply Lemma \ref{l:basic1} with \begin{align*} A_0:= \hat{H}_*^{\hat{T}}(LG/ T),  A_1:=F_R \otimes_{R} \hat{H}_*^{\hat{T}}(LG/ T) \\ M_0:=  \hat{H}_*^{\hat{T}}(M)[q^{\pm 1}] ,  M_1:= F_R \otimes_{\mathbf{k}} \hat{H}_*(M^{T})[q^{\pm 1}] ,\end{align*} the map $S_0$ will be $\cS$ and the maps $i_A$, will be the obvious inclusion map.  The map $i_M$ will be the restriction map  $i_{Fix}^*: \hat{H}_*^{\hat{T}}(M)[q^{\pm 1}]  \to \hat{H}_*^{\hat{T}}(M^{T})[q^{\pm 1}]  \cong R\otimes_\mathbf{k} \hat{H}_*(M^{T})[q^{\pm 1}] $ followed by tensoring.  This map is injective by \cite[\S 7]{Atiyah-Bott}.  \vskip 5 pt

\emph{Step 1:} The first step is to construct a map: \begin{align*} \cS^{F}:  (F_R \otimes_R \hat{H}_*^{\hat{T}}(LG/ T)) \otimes_{\mathbf{k}} (F_R \otimes_{\mathbf{k}} \hat{H}_*(M^{T})[q^{\pm 1}] ) \to F_R \otimes_{\mathbf{k}} \hat{H}_*(M^{T})[q^{\pm 1}]  \end{align*}
so that the diagram commutes.  To define this,  use $F_R \otimes_{R} \hat{H}_*^{\hat{T}}(LG/ T) \cong \bigoplus_{\sigma [w]}F_R \cdot e_{\sigma [w]}$,  to view this as \begin{align} \label{eq:fix1} \cS^{F}:  \bigoplus_{\sigma [w]} (F_R \cdot e_{\sigma [w]} \otimes_\mathbf{k} (F_R \otimes_{\mathbf{k}} \hat{H}_*(M^{T})[q^{\pm 1}] ) \to F_R \otimes_{\mathbf{k}} \hat{H}_*(M^{T})[q^{\pm 1}]  \end{align}

Define the map on each $\sigma [w]$-summand of \eqref{eq:fix1} by $\cS_{\sigma [w]}^{F}$ from \eqref{eq:sigmaw}.    
To check commutativity of the diagram,  note that,  on each summand,  $\cS^F(id \otimes i_M)$ is equivalent to $\cS_{\sigma [w]}$ now viewed as a map $H^*_{\hat{T}}(M) \to F_R\otimes_{R} H^*_{\hat{T}}(M).$ It follows that  $\cS^F(id \otimes i_M)$ is equivalent to just tensoring $\cS$ on the left to obtain $$F_R \otimes_R \cS: F_R \otimes_{R} \hat{H}_*^{\hat{T}}(LG/ T) \otimes_\mathbf{k}  \hat{H}_*^{\hat{T}}(M)[q^{\pm 1}]  \to F_R \otimes_{R} \hat{H}_*^{\hat{T}}(M)[q^{\pm 1}].$$ From here,  it follows that $\cS^F(i_A \otimes i_M)=i_M \circ \cS.$ \vskip 5 pt

\emph{Step 2:} It remains to check that $\cS^F$ is a module map.  Because each $\cS_{\sigma [w]}^{F}$ is $\sigma w$-twisted,  it follows that for any $f \in F_R$,  \begin{align*} \cS_{\sigma [w]}^{F} \circ \cS_f^{F} = \cS_{(\mathcal{A}_{\sigma [w]}(f))}^{F} \circ \cS_{\sigma w}^{F} = \cS_{(\mathcal{A}_{\sigma [w]}(f) \cdot \sigma [w])}^{F}= \cS_{\sigma [w] \cdot f}^{F} \end{align*}

It therefore suffices to check that \begin{align*} \cS_{\sigma_1 [w_1]}^{F} \circ \cS_{\sigma_2 [w_2]}^{F} = \cS_{(\sigma_1 [w_1]) \cdot (\sigma_2 [w_2])}^{F}.  \end{align*} Because these operators are equivalent to the corresponding $F_R$ linear extensions from \eqref{eq:FRotimes},  it suffices to check $\cS_{\sigma_1 [w_1]} \circ \cS_{\sigma_2 [w_2]} = \cS_{(\sigma_1 [w_1]) \cdot (\sigma_2 [w_2])}.$ This reduces this to the case of abelian shift operators studied in \cite{LJ21} by noting that \begin{align*} \cS_{\sigma_{1} [w_{1}]} \circ \cS_{\sigma_2 [w_2]} &= \cS_{w_{1}(w_{1}^{-1} \sigma_{1} [w_{1}])} \circ \cS_{\sigma_{2} [w_{2}]} \\&=\cS_{[w_{1}]} \circ \cS_{(w_1^{-1} \sigma_1 [w_1])} \circ \cS_{\sigma_2} \circ \cS_ {[w_{2}]} \\ &= \cS_{[w_{1}]} \circ \cS_{(w_1^{-1} \sigma_1 [w_1])\cdot \sigma_2} \circ \cS_ {[w_{2}]} \\&= \cS_{(\sigma_1 [w_1])\cdot (\sigma_2 \cdot [w_{2}])}  \end{align*}


In the second and fourth lines we have used Lemma \ref{ex:Weyl}.  The third line is where we have used the module property of abelian shift operators \cite{LJ21}.  
  \end{proof} 


\subsubsection{Quantum connection}\label{s:connection}
We next examine the interplay between our shift operators and the quantum connection (compare \cite[Cor.  3.15]{Iritani} or  \cite[Theorem 1.6]{LJ21}).  For simplicity,  we consider a one-parameter connection along the direction $-c_1^{\hat{T}}(M) \in  H^2_{\hat{T}}(M)$ (an equivariant version of the ``the anti-canonical line" in the terminology of \cite{Galkin-Golyshev-Iritani}).  

 The character $\hat{T}=S^1 \times T \to S^1$ coming from the projection to the first factor induces an algebra map
$H^*(BS^1) \simeq \mathbf{k}[\hat{y}] \to H^*(B\hat{T})$ and we denote the image of the positive generator
by $u \in H^*(B\hat{T})$.
For elements in $A:=H^*(B\hat{T})[q^{\pm 1}] $,  we define
$u q\frac{d}{d q}: A \to A$ to be the $H^*(B\hat{T})$-linear map such that
\begin{align*}
u q\frac{d}{d q}(q^ny):= n q^n u y
\end{align*}
 for any $y \in H^*(B\hat{T})$.
Recall that $QH^*_{\hat{T}}(M):=H^*_{\hat{T}}(M) [q^{\pm 1}]=A \otimes_{H^*(B\hat{T})} H^*_{\hat{T}}(M)$.
By slight abuse of notation, we denote $u q\frac{d}{d q} \otimes id: QH^*_{\hat{T}}(M) \to QH^*_{\hat{T}}(M)$ by $u q\frac{d}{d q}$.

For a class $\beta \in H^*_{\hat{T}}(M)$, there is a well-defined quantum multiplication $\beta \ast_{QH}: QH^*_{\hat{T}}(M) \to QH^*_{\hat{T}}(M)$ (see \cite[Section 3.10]{LJ21}).  
The form $\omega$ is $G$-invariant and the space of $G$-invariant compatible almost complex structures $J$ is contractible (in particular non-empty) \cite[Example D.12]{Ginzburg4}.  Choose one such $J$ so that the determinant $det_\mathbb{C}(TM)$ becomes an $G$-equivariant complex line bundle.  It therefore has an equivariant first Chern class $c_1^{S^1 \times G}(M)$, which restricts to a Weyl-invariant class $c_1^{\hat{T}}(M) \in H^2_{\hat{T}}(M).$ Note that $c_1^{S^1 \times G}(M)$ and $c_1^{\hat{T}}(M)$ are well-defined by contractibility of the space of such $J$.  
The quantum connection is defined by 
\begin{align*}
\nabla_{q\partial_q}= u q \frac{d}{d q}+ -c_1^{\hat{T}}(M) \ast_{QH}:  QH^*_{\hat{T}}(M) \to QH^*_{\hat{T}}(M)
\end{align*}
where  as above $\ast_{QH}$ denotes the equivariant quantum product. 

\begin{rem} Note that this differs from the standard quantum connection by multiplication by $u$ and thus only gives a connection in the direction $u q \frac{d}{d q}.$ On the other hand it is well-defined on $QH^*_{\hat{T}}(M)$ without inverting $u$.   \end{rem} 



\begin{thm}\label{t:commute}
For any $a \in H_*^{\hat{T}}(LG/T)$, the commutator $[\cS(a,-), \nabla_{q\partial_q}]$ vanishes.
\end{thm}
\begin{proof}
Recall that $R=H^{-*}(B\hat{T})$ and $F_R$ is the field of fraction.
Similar to above, we consider
\begin{align*}
F_R \otimes \cS(-,-):& F_R \otimes_R H_*^{\hat{T}}(LG/T) \otimes QH^*_{\hat{T}}(M) \to F_R \otimes_R QH^*_{\hat{T}}(M)  \\
F_R \otimes \nabla_{q\partial_q}:& F_R \otimes_R QH^*_{\hat{T}}(M) \to F_R \otimes_R QH^*_{\hat{T}}(M)
\end{align*}
and it suffices to show that $F_R \otimes \cS(a,\nabla_{q\partial_q} y)- (F_R \otimes \nabla_{q\partial_q})( (F_R \otimes \cS)(a, y))=0$ for all $a \in  F_R \otimes_R H_*^{\hat{T}}(LG/T)$ and $y \in QH^*_{\hat{T}}(M)$.  As above,  it suffices to consider the case where the operator is given by $\cS_{\sigma [w]}^{F}$ from \eqref{eq:sigmaw}.  $\cS_{[w]}$ commutes with the connection because the Weyl action commutes with the quantum product (simply because one can use the symplectomorphism to pushforward the almost complex structures/holomorphic spheres).  

Thus,  we are again reduced to the case of a co-character $\cS_{\sigma}$ as in  \cite[Theorem 3.6]{LJ21}.  As strictly speaking the setup for Theorem 3.6 of \emph{loc.  cit} is slightly different we explain how to adapt the proof to our setting.   Let $E(\sigma)$ be the Seidel space and let $c_{1, vert}^{\hat{T}} \in H^2_{S^1 \times T}(E(\sigma))$ be the $\hat{T}$-equvariant first Chern class of the vertical tangent bundle.  Then the intertwining relation of \cite[Theorem 3.4]{LJ21} says that for any $y \in QH^*_{\hat{T}}(M)$,  
\begin{align} \label{eq:intertwining1} \mathcal{C}_{\sigma}(\cQ(y)\ast (c_{1, vert}^{\hat{T}})_{|{0}})- \mathcal{C}_{\sigma}(\cQ(y)) \ast (c_{1, vert}^{\hat{T}})_{|\infty}= u\mathcal{WC}_{\sigma}(\cQ(y),  c_{1, vert}^{\hat{T}}) \end{align} 
where $(c_{1, vert}^{\hat{T}})_{|{0}}$(respectively $(c_{1, vert}^{\hat{T}})_{|\infty}$ is the restriction of the vertical first Chern class to the divisor at 0 (respectively $\infty$).  The operator $\mathcal{WC}_{\sigma}(\cQ(y),  c_{1, vert}^{\hat{T}})$ is a Seidel map where curves in section class $A$ are weighted by $c_{1, vert}^{\hat{T}}(A)=c_{1, vert}(A)$.  We have that $(c_{1, vert}^{\hat{T}})_{|\infty}=c_1^{\hat{T}}(M)$ and $(c_{1, vert}^{\hat{T}})_{|{0}}=\cQ(c_1^{\hat{T}}(M))$.  Using this together with the fact that $\cQ$ commutes with quantum product,  we can rewrite \eqref{eq:intertwining1} as 
\begin{align} \label{eq:intertwining2} \mathcal{S}_{\sigma}(y \ast c_{1}^{\hat{T}}(M))-\mathcal{S}_\sigma(y) \ast c_{1}^{\hat{T}}(M)= u\mathcal{WC}_{\sigma}(\cQ(y),  c_{1, vert}^{\hat{T}}) \end{align} 
On the other hand from the definition of $\mathcal{WC}_{\sigma}(\cQ(y),  c_{1, vert}^{\hat{T}})$,  we have that 
\begin{align} \label{eq:intertwining3} uq\frac{d}{dq} (\mathcal{S}_{\sigma}(y))-\mathcal{S}_\sigma(uq\frac{d}{dq}y)= u\mathcal{WC}_{\sigma}(\cQ(y),  c_{1, vert}^{\hat{T}}) \end{align} 
The fact that the right hand side of this equation and \eqref{eq:intertwining2} are the same implies the commutation with the connection.  
\end{proof} 



\subsection{A general Peterson map} 

We next consider the based-loop space $\Omega G$.  This admits a natural action of $\hat{G}:=S^1 \times G$ given by \begin{align} \label{eq:conjug} G\times \Omega G \to \Omega
  G  \\ g \cdot \gamma(t) = g\gamma(t)g^{-1}.  \nonumber
\end{align} 

\begin{align} \label{eq:loopy} S^1 \times \Omega G \to
  \Omega G  \\ a \cdot \gamma(t) =
  \gamma(t-a)\gamma(-a)^{-1}.  \nonumber 
\end{align} 

The space $\Omega G$ admits a Pontryagin product given by pointwise multiplication:
\begin{align} \label{eq:mG} m_{\Omega G}: \Omega G \times \Omega G \to \Omega G \end{align}
This map \eqref{eq:mG} is manifestly $G$-equivariant if $\Omega G \times \Omega G$ is given the diagonal $G$-action.  It therefore induces a map \begin{align} \label{eq: Ghatmg} \hat{m}_{G}: \hat{H}_*^G(\Omega G) \otimes \hat{H}_*^G(\Omega G) \to \hat{H}_*^G(\Omega G)  \end{align} 
where $\hat{m}_{G}$ is the composition of $m_{\Omega G,*}$ with the restriction along the diagonal subgroup and the K{\"u}nneth map.  This restricts to a $T$-equivariant Pontryagin product 
\begin{align} \label{eq: Thatmg} \hat{m}_{T}: \hat{H}_*^T(\Omega G) \otimes \hat{H}_*^T(\Omega G) \to \hat{H}_*^T(\Omega G)  \end{align} 

As before,  the homology groups are concentrated in even degrees.  Moreover,  it is well-known that \eqref{eq:mG} and \eqref{eq: Thatmg} define (graded-) commutative structures--- this can be seen for example by rewriting $\Omega G$ as a double loop space $\Omega G \cong \Omega^2 BG$ (this is easily checked to be compatible with Pontryagin products and $G$-actions).  There is a natural $T$-equivariant map $j: \Omega G \to LG/T$,  which induces a map: \begin{align*} j_*: \hat{H}_*^{T}(\Omega G) \to \hat{H}_*^{T}(LG/T).  \end{align*}

The map $j_*$ is easily verified to be a homomorphism of $H^*(BT)$ algebras (see  \cite[Lemma 4.3]{Lam}).  Let \begin{align} \label{eq:peterson} \mathcal{P}: \hat{H}_*^{T}(\Omega G) \to H_T^*(M) \\ \alpha_0 \to \mathcal{S}_{j_{*}(\alpha_0)}^{u=0}([M]_T) \nonumber \end{align} 

We will again need a simple algebraic fact: 

\begin{lem} \label{lem:tensorhom} Suppose $i_A: A_0 \hookrightarrow A_1$ and $i_B: B_0 \hookrightarrow B_1$ are embeddings of $\mathbf{k}$-algebras.  Let $F_1: A_0 \to B_0$ and  $F_2: A_1 \to B_1$ be maps of $\mathbf{k}$-modules such that the following diagram commutes:\[
\begin{tikzcd}
A_0  \arrow{r}{F_1} \arrow[swap]{d}{i_A}& B_0 \arrow{d}{i_B} \\
A_1 \arrow{r}{F_2} & B_1
\end{tikzcd}
\]
 If $F_2$ is a ring homomorphism,  then so is $F_1$.  \end{lem} 
\begin{proof} 
This again follows from diagram chasing:
\begin{align*}
i_B(F_1(a_1a_2))=&F_2(i_A(a_1a_2))=F_2(i_A(a_1)i_A(a_2))\\
=&F_2(i_A(a_1)) \cdot F_2(i_A(a_2))=i_B(F_1(a_1)) \cdot i_B(F_2(a_2))\\
=&i_B(F_1(a_1) \cdot F_2(a_2)).
\end{align*}
The injectivity of $i_B$ implies that $F_1(a_1a_2)=F_1(a_1) \cdot F_2(a_2)$.
\end{proof}

\begin{lem}\label{l:Peterson}
 Equip  $\hat{H}_*^{T}(\Omega G)$ with the $T$-equivariant Pontryagin product \eqref{eq: Thatmg} and $H_T^*(M)$ with the equivariant quantum product.  The map \eqref{eq:peterson} becomes a ring homomorphism.   \end{lem} 
\begin{proof} We set $\bar{R}=H^*(BT)$  and $F_{\bar{R}}$ be its fraction field.   We apply Lemma \ref{lem:tensorhom} with the following choices:  \begin{align*}A_0:= \hat{H}_*^{T}(\Omega G),  A_1:= F_{\bar{R}} \otimes_{\bar{R}} \hat{H}_*^{T}(\Omega G) \\
 B_0:= QH^*_T(M),  B_1:= F_{\bar{R}} \otimes_{\bar{R}} QH^*_T(M) 
 \end{align*} The maps $i_A$,  $i_B$ are the obvious inclusions and the map $F_1$ is the map $\mathcal{P}$ defined in \eqref{eq:peterson}.  Note that it follows from the fact that $j_*$ is $\bar{R}$-linear that $\mathcal{P}$ is $\bar{R}$-linear.  We set $F_2$ to be the $F_{\bar{R}}$-linear extension of $\mathcal{P}$. 

Having made these choices,  it suffices to prove that $F_2$ is a ring homomorphism.  $F_{\bar{R}} \otimes_{\bar{R}} \hat{H}_*^{T}(\Omega G)$ has a basis given by fixed points of the $T$-action on $\Omega G$,  which are given by co-characters $\sigma \in \mathcal{X}(T)$ viewed as loops in $\Omega G$.  Pushing this forward along $j$ simply sends this to same loop $\sigma$,  now viewed as lying in $LG/T$ and the element $\mathcal{S}_{\sigma}^{u=0}([M]_T)$ is the equivariant Seidel element \cite{LJ21} corresponding to this co-character.  It is well-known c.f.  \emph{loc. cit} that for any two co-characters,  $\sigma_1, \sigma_2$,  \begin{align*}  \mathcal{S}_{\sigma_1}^{u=0}([M]_T) \ast_{QH} \mathcal{S}_{\sigma_2}^{u=0}([M]_T)=  \mathcal{S}_{\sigma_1 \cdot \sigma_2}^{u=0}([M]_T).  \end{align*}     \end{proof}

\section{The Lagrangian $\mathbb{L}_G(M)$}\label{section: other}

Throughout this section,  we take our ground field $\mathbf{k}$ to a field of characteristic zero.  If we view $\Omega G$ as $LG/G$,  the convolution construction from \S \ref{sec:conv} gives rise to an associative algebra structure which was heavily studied in \cite{BFM}: 
\begin{align} \label{eq: hatmg} \hat{m}: \hat{H}_*^{\hat{G}}(\Omega G) \otimes \hat{H}_*^{\hat{G}}(\Omega G) \to \hat{H}_*^{\hat{G}}(\Omega G)  \end{align} 

To connect this story with the main story in this paper,  we need the following observation: 

\begin{thm}[\cite{KK1,KK2,Ku}, see also Theorem 1.4.1 of \cite{Ginzburg18}] \label{ref:Ginzburg}
 There is an embedding of algebras $\hat{H}_*^{\hat{G}}(\Omega G) \hookrightarrow \hat{H}_*^{\hat{T}}(LG/T)$.
It identifies $\hat{H}_*^{\hat{G}}(\Omega G)$ as the spherical subalgebra of $\hat{H}_*^{\hat{T}}(LG/T)$
with respect to the symmetrizer idempotent $\mathbf{e}=\frac{1}{|W|}\sum_{w \in W}w \in \mathbf{k}[W]$. \end{thm} 

The other basic property of the algebra $\hat{H}_*^{\hat{G}}(\Omega G)$ that we will need is that it is a flat deformation of the Pontryagin ring $\eqref{eq: Ghatmg}.$

\begin{thm}[\cite{BFM}]  Identify $H^*(BS^1) \cong \mathbf{k}[u].$ \begin{itemize} \item The homology $\hat{H}_*^{G}(\Omega G)$ equipped with the product \eqref{eq: hatmg} is concentrated in even degrees and is strictly commutative.  \item There is an additive identification of $H^*(BS^1)$ modules \begin{align} \label{eq:additivelytrivial} \hat{H}_*^{\hat{G}}(\Omega G) \cong \hat{H}_*^{G}(\Omega G)[u]. \end{align} \item The  $u=0$ limit  $\hat{H}_*^{\hat{G}}(\Omega G) \otimes_ {\mathbf{k}[u]}\mathbf{k}$ is isomorphic to the Pontryagin ring \eqref{eq: Ghatmg}.   \end{itemize} 
 \end{thm} 

\begin{proof} (Sketch) These results are all proven in \cite{BFM},  but because they are elementary we indicate how they are proven.  As noted just below \eqref{eq: Thatmg},  the first bullet point is classical.  The additive identification in the second bullet follows from the fact that $\hat{H}_*^{G}(\Omega G)$ is concentrated in even degree together with standard spectral sequence arguments.  The claim about the $u=0$ limit amounts to the fact that we have a $G$-equivariant identification $LG \times_G (LG/G) \cong \Omega G \times \Omega G. $ \end{proof}

 \begin{rem} It is worth noting one potential point of confusion --- the induced $S^1$-action on $\Omega G \times \Omega G$ coming from identifying $LG \times_G (LG/G) \cong \Omega G \times \Omega G$ is \emph{not} the diagonal $S^1$-action.  (The map \eqref{eq:mG} is not equivariant with respect to this diagonal $S^1$-action. ) \end{rem}

We will need to make use of a celebrated result of Gabber
concerning modules over quantized algebras.   In the
discussion below,  let $A$ be an algebra over
$\operatorname{D}= \mathbf{k}[u]/(u^2)$ which is free as a module over $\operatorname{D}$.  Suppose that $A/uA=A_0$ is a smooth,  finitely generated \emph{commutative} algebra over a field $\mathbf{k}$.  It is well-known that $A_0$ comes equipped with a Poisson bracket defined as follows.  Given any two elements $a_0,b_0 \in A_0$ take lifts $a,b$ to $A$ and form the commutator $[a,b] \in A/uA \cong A_0$.  It is easy to see that this is well-defined and satisfies the Leibnitz identity.  It is worth mentioning that in general,  these Poisson brackets satisfy the Leibnitz identity but not necessarily the Jacobi identity.  However,  the Poisson bracket will also satisfy the Jacobi identity if the deformation can be extended modulo $u^3$.  

\begin{defn} Let $Z \subset Y$ be a reduced subvariety and let $I_Z \subset A_0$ be the corresponding ideal sheaf.  The subvariety $Z$ is said to be co-isotropic if $\lbrace I_Z, I_Z \rbrace \subset I_Z.$ \end{defn} 

If the Poisson structure is dual to a symplectic structure,  then $Z \subset Spec(A_0)$ being co-isotropic is equivalent to it being co-isotropic in the usual sense along $Z_{reg}$.  In this case,  we say that a subvariety $Z$ is Lagrangian if it is co-isotropic and of minimal dimension $\operatorname{dim}(Y)/2$. 

 \begin{thm}[\cite{Gabber}, see also Theorem 1.2.8 of \cite{GinzburgLecture}] \label{thm: Gabber} 
 Let $M$ be a finitely generated module over $A$,  which is also free over $\operatorname{D}.$  Finally set $I_M=\sqrt{Ann_{A_{0}}(M_0)}$  to be the radical of the annhilator ideal $Ann_{A_{0}}(M_0).$   Then $$\lbrace I_M,I_M \rbrace \subset I_M.$$ In other words,  the (reduced) support of $M_0$ over $A_0$ is co-isotropic.  
 \end{thm}

For the present paper,  the key example will be the following:\footnote{In \cite{BFM},  the authors assume semi-simplicity.  However this is not needed,  see e.g. \cite[Appendix A]{BFN}.}

\begin{thm}[\cite{BFM}] We have
\begin{itemize} 
\item The spectrum $\operatorname{Spec}(\hat{H}_*^{G}(\Omega G,\mathbb{C}))$ is a smooth holomorphic symplectic manifold.   \item The infinitesimal deformation given by restricting \eqref{eq:additivelytrivial} modulo $u^2$ is induced by a Poisson structure dual to this symplectic structure.   \end{itemize}  \end{thm} 

We are now in a position to expand upon \cite[Remark 2.3]{Teleman1}: 

\begin{proof}[Proof of Corollary \ref{c:LagrangianSupport}]
The quantum cohomology $QH^*_T(M)$ becomes a module over $\hat{H}_*^{\hat{G}}(\Omega G)$ by Theorem \ref{ref:Ginzburg}.  This module structure preserves $QH^*_G(M)=QH^*_T(M)^{W}$ because $\hat{H}_*^{\hat{G}}(\Omega G)$ is the spherical subalgebra with respect to the symmetrizer idempotent. The module $QH^*_G(M,\mathbb{C})$ has co-isotropic support $\mathbb{L}_G(M)$ over $\operatorname{Spec}(\hat{H}_*^{G}(\Omega G,\mathbb{C}))$ by Theorem \ref{thm: Gabber}.  The support is Lagrangian because $QH^*_G(M)$ is a finite module over $H^*(BG)$
and hence $\dim( \operatorname{Spec}(QH^*_G(M)))=\dim( \operatorname{Spec}(H^*(BG))) =\frac{1}{2} \dim( \operatorname{Spec}(\hat{H}_*^{G}(\Omega G)) )$. \end{proof}

We close the paper with a couple of calculations from the literature which illustrate this corollary.  \vskip 10 pt

\emph{Toric varieties:}  Let $M^{2n}$ be a compact,  monotone toric manifold.  In this case,  $$BFM(G_{\mathbb{C}}^{\vee}) \cong T^*T^{\vee}_\mathbb{C}.$$  We view this as $\operatorname{Spec}(\mathbb{C}[h_1,\cdots, h_n,  z_1^{\pm 1},\cdots,z_n^{\pm 1}])$ with the holomorphic symplectic form $\sum_{i=1}^n dh_i \wedge \frac{dz_i}{z_i}$ where the coordinates $h_i$ correspond to the the equivariant variables and the $z_i$ correspond to the Seidel operators.  Let $$W_{HV}(z_1,\cdots,z_n):(\mathbb{C}^*)^n \to \mathbb{C} $$ denote the Givental-Hori-Vafa superpotential (see \cite[Theorem 3]{MR1403947}) and consider its equivarant version $$W_{HV}^{eq}:= W_{HV}+\sum_i h_i\operatorname{log}(z_i).$$ Then it follows from \cite{LJ21, MR3651574} that  there is an isomorphism of rings:  $$ QH_T^*(M) \cong \operatorname{Jac}_{/\mathbf{k}[h_i]}(W_{HV}^{eq}).$$  Here the notation  $\operatorname{Jac}_{/\mathbf{k}[h_i]}(W_{HV}^{eq})$ denotes the relative Jacobian ring of $W_{HV}^{eq}$,  meaning we only take partial derivatives with respect to the $z_i$ variables.  Writing this out explicitly,  this is the subvariety of $\operatorname{Spec}(\mathbb{C}[h_1,\cdots, h_n,  z_1^{\pm 1},\cdots,z_n^{\pm 1}])$  defined by the equations \begin{align} -z_i\frac{\partial W_{HV}}{\partial z_i}= h_i \end{align} which is Lagrangian because it is  essentially the graph of the differential of $W_{HV}$ written out in coordinates.   \vskip 10 pt

\emph{Flag varieties:} 
(compare \cite[\S 6.2]{Teleman2})
This example is strictly speaking conjectural.  Suppose $M=G/T$ is a full flag variety.  The description of the Lagrangian in this case involves an alternative realization of $BFM(G_{\mathbb{C}}^{\vee})$ as a (bi-Whittaker) Hamiltonian reduction of $T^*G_{\mathbb{C}}^{\vee}$.  Let $\rho: \mathfrak{sl}(2,\mathbb{C}) \to \mathfrak{g}^{\vee}_\mathbb{C}$ be a Lie algebra homomorphism such that $y$ (the image under $\rho$ of the standard lower triangular generator) is a principal nilpotent element in $\mathfrak{g}_{\mathbb{C}}^{\vee}.$  Let $N^{\vee}_{\mathbb{C}}$ be the unipotent subgroup of $G_{\mathbb{C}}^{\vee}$ whose Lie algebra $\mathfrak{n}^{\vee}_{\mathbb{C}}$
is the sum of negative eigenspaces of $h$ (the image under $\rho$ of the standard diagonal generator).  View $y$ as lying in $\mathfrak{n}^{\vee,\ast}$ via an invariant pairing on $\mathfrak{g}_{\mathbb{C}}^{\vee}.$ Let $\mu$ be the moment map for the $N^{\vee} \times N^{\vee}$ action on $T^*G_{\mathbb{C}}^{\vee}.$ Then we have an isomorphism: \begin{align} BFM(G_{\mathbb{C}}^{\vee}) \cong \mu^{-1}(y)/N^{\vee} \times N^{\vee}  \end{align}

If $w_0$ is a longest word in the Weyl group,  the inclusion of $N^{\vee} \times w_0 T_\mathbb{C}^{\vee} \times N^{\vee} \subset G_\mathbb{C}^{\vee}$ leads to a (symplectic) embedding of $T^*T_\mathbb{C}^{\vee} \subset BFM(G_{\mathbb{C}}^{\vee}).$ Then we expect that $\mathbb{L}_G(G/T)$ is given by a cotangent fiber in $T^*T_\mathbb{C}^{\vee}.$ In fact,  under the identification of \cite[Theorem 2.12]{BFM},  the induced projection $T^*T_\mathbb{C}^{\vee} \to \operatorname{Spec}(H^*(BG))$ becomes identified with the Toda integrable system.  The classical calculations of \cite{GiventalKim93,  Kim} describe the equivariant quantum cohomology $QH_G^*(G/T)$ (as an $H^*(BG)$ module) precisely as a cotangent fiber in this Toda integrable system.  To complete this example,  it would suffice to show that these calculations are compatible with our geometric description of the module structure over $BFM(G_{\mathbb{C}}^{\vee}).$ While such comparisons would take us too far afield,  we expect that this could be verified using the methods of \cite{ChiHong3}.

\def\cprime{$'$}

{\small

\medskip
\noindent Eduardo Gonz\'alez \\
\noindent University of Massachusetts, Boston, 100 William T, Morrissey Blvd, Boston, MA 02125, US\\
{\it e-mail:} Eduardo.Gonzalez@umb.edu
\medskip

\medskip
\noindent Cheuk Yu Mak\\
\noindent School of Mathematics, University of Edinburgh, James Clerk Maxwell Building, Edinburgh, EH9 3FD, UK\\
{\it e-mail:} cheukyu.mak@ed.ac.uk

\medskip
 \noindent Dan Pomerleano\\
\noindent University of Massachusetts, Boston, 100 William T, Morrissey Blvd, Boston, MA 02125, US\\
 {\it e-mail:} Daniel.Pomerleano@umb.edu

}

\end{document}